\documentclass[11pt,reqno ]{amsart}
\usepackage{}
\usepackage{mathrsfs}
\usepackage{amsmath,amsfonts,amssymb,amscd,amsthm,bbm}
\usepackage{graphicx}
\usepackage{color}
\usepackage{ifpdf}
\newcommand{\ua}{u^\alpha}
\newcommand{\ub}{u^\beta}

\newcommand{\Tab}{T^{\alpha\beta}}

\title[Classical limit  of the  relativistic Cucker-Smale flocking]{Optimal convergence speed in the classical limits of relativistic Cucker-Smale models}

\author[Ha]{Seung-Yeal Ha}
\address[Seung-Yeal Ha]{\newline Department of Mathematical Sciences and Research Institute of Mathematics\newline
    Seoul National University, Seoul 08826, Republic of Korea}
\email{syha@snu.ac.kr}

\author[Ruggeri]{Tommaso Ruggeri}
\address[Tommaso Ruggeri]{\newline Department of Mathematics and Alma Mater Research Center on Applied Mathematics AM$^2$,
University of Bologna, Bologna, Italy}
\email{tommaso.ruggeri@unibo.it}

\author[Xiao]{Qinghua Xiao}
\address[Qinghua Xiao]{\newline  Innovation Academy for Precision Measurement Science and Technology, Chinese Academy of Sciences,  Wuhan 430071, China}
\email{xiaoqh@apm.ac.cn}

\newtheorem{theorem}{Theorem}[section]
\newtheorem{lemma}{Lemma}[section]
\newtheorem{corollary}{Corollary}[section]
\newtheorem{proposition}{Proposition}[section]
\newtheorem{remark}{Remark}[section]

\newtheorem{definition}{Definition}[section]

\newcommand{\bbr}{\mathbb R}

\numberwithin{equation}{section}

\begin{document}
%%%%%%%%%%%%%%%%

\date{}

%\subjclass{    } \keywords{Relativistic Landau-Maxwell system, Vlasov-Poisson-Landau system, uniform $L^2$-stability, equilibrium}
%\thanks{\textbf{Acknowledgment.} The work of...}
\thanks{Achnowledgment: The work of S.-Y. Ha was supported by NRF-2020R1A2C3A01003881 and the LAMP Program of the National Research Foundation of Korea (NRF) grant funded by the Ministry of Education (No. RS-2023-00301976). The work of Q. H. Xiao was supported by the National Natural Science Foundation of China grant 12271506 and the National Key Research and Development Program of
China (No.2020YFA0714200). The work of T. Ruggeri was partially supported and carried out in the framework of the activities of the Italian National Group for Mathematical Physics [Gruppo Nazionale per la Fisica Maematica (GNFM/IndAM)]}

\begin{abstract} We study quantitative estimates for the flocking and uniform-time classical  limit to the relativistic Cucker-Smale (in short RCS) model introduced in \cite{Ha-Kim-Ruggeri-ARMA-2020}. Different from previous works, we do not neglect the relativistic effect on the presence of the pressure in momentum equation. For the RCS model, we provide a quantitative estimate on  the uniform-time classical limit with an optimal convergence rate which is the same as in finite-time classical limit under a relaxed initial condition. We also allow corresponding initial data for the RCS and Cucker-Smale (CS) model  to be different in the classical limit. This removes earlier constraints employed in the previous classical limit. As a direct application of this optimal convergence rate in the classical limit of the RCS model, we derive an optimal convergence rate for the corresponding uniform-time classical limit for the kinetic RCS model. 
 \end{abstract}

\maketitle
\date{}
\centerline{\date}
 \begin{center}
Dedicated to Prof. Shih-Hsien Yu on his 60th birthday, with friendship and admiration.    

\vspace{1cm}
 \end{center}
%\tableofcontents

\section{Introduction} \label{sec:1}
\setcounter{equation}{0}
Collective behaviors of interacting particle systems are often observed in our daily life, to name a few, aggregation of bacteria \cite{T-B-L, T-B}, flocking of birds \cite{C-S, R, H-T}, swarming of fish \cite{Ao, D-M2, T-T}, synchronization of fireflies and pace-maker cells \cite{B-B, Ku2, Pe, W2} and herding of sheep, etc. Among them, we are interested in the flocking behaviors in which all agents move with the same velocity using a simple interaction rule between particles. In literature, several mathematical models have been proposed to model flocking behaviors. For example, Reynolds introduced a simple simulation model \cite{R} for the flocking of boids based on three rules (repulsion, velocity alignment and attraction), and then Vicsek introduced a discrete dynamical system \cite{V-C-B-C-S}  and showed that four different dynamics can emerge from given initial data by varying the number of particles in a fixed computation domain and changing a noise strength. Later, this simulation result was further analyzed in \cite{J-L-M} based on the infinite assumptions on connectivity. 

In 2007, Cucker and Smale (CS) introduced an analytically tractable particle model \cite{C-S} based on Newton's law of motion, where each particle experiences an internal force determined by the weighted sum of relative velocities. For an overview of this topic, we refer to books and survey articles \cite{A-B-F, B-H, D-B1, P-R, VZ, Wi1}.
Ha and Ruggeri \cite{Ha-Ruggeri-ARMA-2017} were the first to observe that the momentum equation in the CS model is equivalent to that of a gas mixture, under the assumption of isothermal processes and spatially homogeneous solutions (i.e., solutions that depend only on time). This observation allowed them to generalize the Cucker-Smale model by incorporating the effects of distinct energy balance laws for each species, thereby accounting for thermodynamic effects. This generalized model was named as the TCS model, an acronym for the "{\it{thermodynamic Cucker-Smale model}}". This model was further generalized to a relativistic framework by Ha, Kim, and Ruggeri \cite{Ha-Kim-Ruggeri-ARMA-2020}. In this context, the authors introduced a relativistic counterpart of the TCS model, still using the analogy with a mixture of gases and focusing on spatially homogeneous solutions. Then, using the method of \emph{principal sub-systems}, which was introduced for hyperbolic systems by Boillat and Ruggeri \cite{B-R}, it is possible to consider a simplified \emph{mechanical} model of the previous system \cite{Ha-Kim-Ruggeri-ARMA-2020}. This model becomes a natural generalization of the relativistic Cucker-Smale model (RCS). The flocking and classical limit of this simplified model have been extensively studied in a series of works \cite{Ahn-Ha-Kim-JMP-2022,Ahn-Ha-Kim-CPAA-2021,Ha-Kim-Ruggeri-ARMA-2020,Ha-Kim-Ruggeri-CMS-2021}. In the relativistic case, the momentum equation also includes the time derivative of the pressure as a relativistic effect. In all previous studies, pressure effect was ignored. 

The purpose of this paper is to account for the relativistic effect of pressure in RCS and to enhance certain aspects of the qualitative analysis compared to previous studies. In particular, the main novelty of this work is the derivation of an optimal convergence rate in the uniform-time classical limit, which matches that of the finite-time classical limit. As a direct application of this optimal convergence rate, we provide a straightforward proof for the uniform-time classical limit of the kinetic RCS model.

The rest of this paper is organized as follows. In Section \ref{sec:2}, we review the relativistic model for gas mixture and entropy principle and study the relativistic TCS model and its reduction to the classical CS model, and we summarize our main results. In Section \ref{sec:3}, we present a refined sufficient framework for asymptotic flocking and derive an optimal convergence speed in the uniform-in-time mean field limit from the relativistic CS model to the CS model. In Section \ref{sec:4}, we consider the  corresponding issues to the kinetic CS model. Finally Section \ref{sec:5} is devoted to a brief summary of main results and some remaining issues for a future work. \newline

%\vspace{0.5cm}

\noindent {\bf Gallery of Notation}: Below, we list notation for physical and thermodynamic observables to be used throughout the paper. Let $a \in [N]$ be any index.
\begin{align*}
	\begin{aligned}
		&  n_a:~\mbox{particle number}, \quad  m_a:~\mbox{mass in rest frame}, \quad \rho_a= n_a m_a:~\mbox{density},  \\
		& \Gamma_a= \frac{1}{\sqrt{1-{v^2_a}/{c^2}}}:~\mbox{the Lorentz factor}, \quad c:~\mbox{speed of light}, \quad v^i_a:~\mbox{velocity}, \\
		& u^\alpha_a \equiv (u^0_a = \Gamma_a c, u^i_a= \Gamma_a v^i_a):~\mbox{four-velocity vector}, \quad  p_a:~\mbox{pressure}, \\
		&  h^{\alpha \beta}_a =  \ua_a \ub_a /{c^2} - g^{\alpha\beta}:~\mbox{the projector tensor},\\
		& g^{\alpha\beta}:~ \mbox{metric tensor with signature $(+,-,-,-)$}, \\
		& \varepsilon_a:~\mbox{internal energy density}, \quad \partial_\alpha = \partial / \partial x_\alpha,~~x^\alpha \equiv (x^0 = c t, x^i):~\mbox{space-time coordinates},
	\end{aligned}
\end{align*}
$	e_a= \rho_a(c^2+\varepsilon_a)$
is the energy-that is, the sum of the energy in the rest frame and  the internal energy. The Greek indices run from $0$ to $4$ and the Latin indices from $1$ to $3$ and, as usual, pairs of lower/upper indices indicate summation. \newline

\section{Descriptions of models and main results}\label{sec:2}
\setcounter{equation}{0}
In this section, we review the relativistic fluid model for a mixture and entropy principle for the proposed fluid model, and the  TCS model according with the paper \cite{Ha-Kim-Ruggeri-ARMA-2020}. Consider $N$ species mixture of relativistic Eulerian gases in $\bbr^3$, and we assume that each constituent in the simple mixture of gases obeys the same balance law as for a single fluid, and there are production terms which take into account of interchanges in momentum and energy between continents, and then summarize our main results.  \newline

\subsection{The Relativistic mixture of gas and relativistic CS model}
For each species $a \in [N]$, let $V^\alpha_a$ and $T^{\alpha\beta}_a$ be the particle-particle flux and energy-momentum tensor \cite{Anile, Cercignani-Kremer, Groot-Leeuwen-Weert-1980, Synge,Taub-1967,book}, respectively:
\begin{align*}%\label{EuVT}
	V^\alpha_a := \rho_a \ua_a \quad \mbox{and} \quad  \Tab_a := p_a h^{\alpha\beta}_a +\frac{e_a}{c^2} \ua_a \ub_a \quad \alpha, \beta \in \{0\}\cup[3].
\end{align*}
The field equations for relativistic single fluid correspond to the balance of particle numbers and energy-momentum tensor in Minkowski space:
\begin{align}\label{massmomenergy}
	& \partial_\alpha V^\alpha_a = \tau_a,  \quad
	\partial_\alpha T^{\alpha\beta}_a=  \mathcal{M}^{\beta}_a,\quad a \in [N].
\end{align}
Here, $\tau_a$ and $\mathcal{M}^{\beta}_a$ represent production terms that satisfy zero-sum constraints, ensuring that the total particle number and energy-momentum of the mixture are conserved:
\begin{align*}\label{zerosumcon}
	\sum_{a=1}^N\tau_a = 0,  \quad
	\mathcal{M}^{\beta}_a=0,\quad \beta \in \{0 \} \cup [3].
\end{align*}
Since $\tau_a$ is  due to chemical reactions, for simplicity, we may assume that
$$\tau_a=0, \quad a \in [N]. $$
The relations \eqref{massmomenergy} can be rewritten as balanced equations in space-time:
\begin{equation}
	\begin{cases}  \label{constituteaeqn}
		\displaystyle \partial_ t\left(\rho_a \Gamma_a\right) +  \partial_i\left(\rho_a \Gamma_av^i_a\right)=0, \quad a\in[N],\\
		\displaystyle  \partial_ t\left(\Gamma_a^2v^j_a\left(\rho_a+\frac{p_a+\rho_a\varepsilon_a}{c^2}\right)\right) +  \partial_i\left(p_a\delta^{ij}+\Gamma_a^2v^j_av^i_a\left(\rho_a+\frac{p_a+\rho_a\varepsilon_a}{c^2}\right)\right) = M^j_a,  \\
		\displaystyle \partial_ t\big(p_a(\Gamma_a^2-1)+\rho_a\big(\Gamma_a^2\varepsilon_a+c^2\Gamma_a(\Gamma_a-1)\big)\big)  \\
		\displaystyle \hspace{2cm} +\partial_i\big(p_a\Gamma_a^2+\rho_a\big(\Gamma_a^2\varepsilon_a+c^2\Gamma_a(\Gamma_a-1)\big)v^i_a\big) =\mathcal{E}_a,
	\end{cases}
\end{equation}
where the momentum and energy production rates are defined as follows:
\begin{align*}%\label{mea}
	M^i_a:=\mathcal{M}^i_a, \quad \mathcal{E}_a:=c\mathcal{M}^0_a.
\end{align*}
In   \cite{Ha-Kim-Ruggeri-ARMA-2020}, global quantities $\rho, v^j, \varepsilon$, and $p$ are also defined such that the sum of densities in \eqref{constituteaeqn} is the same as a single fluid:
\begin{align}\label{globdens}
	\begin{aligned}
		&\rho\Gamma=\sum_{a=1}^N\rho_a\Gamma_a,\qquad \Gamma=\frac{1}{\sqrt{1-\frac{v^2}{c^2}}},  \\
		&\Gamma^2v^j\left(\rho+\frac{p+\rho\varepsilon}{c^2}\right)
		=\sum_{a=1}^N\left(\Gamma_a^2v^j_a\left(\rho_a+\frac{p_a+\rho_a\varepsilon_a}{c^2}\right)\right), \\
		&p\left(\Gamma^2-1\right)+\rho\big(\Gamma^2\varepsilon+c^2\Gamma(\Gamma-1)\big)\big)\\
		&\hspace{2cm} =\sum_{a=1}^N\left(p_a\left(\Gamma_a^2-1\right)
		+\rho_a\left(\Gamma_a^2\varepsilon_a+c^2\Gamma_a\left(\Gamma_a-1\right)\right)\right).
	\end{aligned}
\end{align}
Then the total pressure $p$, total density $\rho$, the average velocity $\mathbf{v}\equiv(v^j)$, and the total internal energy $\varepsilon$ can be defined in a unique manner. They are reduced to the classical ones in the formal classical limit $c\rightarrow \infty$. The global average velocity $v^j$ can be associated to the global four-velocity
$$U^{\alpha}\equiv \Gamma\left(c, v^i\right),\qquad U^{\alpha}U_{\alpha}=c^2.$$ 
Additional state equations are needed in order to make the system (\ref{constituteaeqn}) be closed. In this paper, we assume that
constitutive functions for $p_a$ and $\varepsilon_a$ are the same of an ideal classical polytropic gas\footnote{In reality, the energy becomes more complex when described using kinetic theory, as it is well known that this theory assumes the so-called Synge expression, which involves the ratio of Bessel functions (see, e.g., \cite{R-X-Z}).}:
\begin{equation}  \label{press}
	p_a= k_B n_aT_a=\frac{k_B}{m_a}\rho_a T_a, \quad \mbox{and} \quad 
	e_a= \rho_a \big(c^2+c_{V_a}T_a\big), \qquad a \in [N],
\end{equation}
where $k_B$ is the Boltzmann constant, and  $c_{V_a}$ is the specific heat that is assumed to be constant (polytropic gas).
\subsubsection{Thermomechanical ensemble} \label{sec:2.1} 
We assume the spatial homogeneity in \eqref{constituteaeqn}, Then, we use \eqref{press} to obtain the corresponding ODE system:
\begin{align}\label{eulerODE}
	\begin{aligned}
		&\frac{\mathrm{d}}{\mathrm{d} t}\left(\rho_a\Gamma_a\right)=0,\quad a \in [N],\\
		&\frac{\mathrm{d}}{\mathrm{d} t}\left(\rho_a\Gamma_a^2\mathbf{v}_a\left(1+\frac{\frac{p_a}{\rho_a}+c_{V_a}T_a}{c^2}\right)\right)
		=\mathbf{M}_a,\\
		&\frac{\mathrm{d}}{\mathrm{d} t}\left[p_a\left(\Gamma_a^2-1\right)+\rho_a\left(\Gamma_a^2c_{V_a}T_a+c^2\Gamma_a\left(\Gamma_a-1\right)
		\right)\right]=\mathcal{E}_a.
	\end{aligned}
\end{align}
It follows from  $\eqref{eulerODE}_1$ that  $\rho_a\Gamma_a$ is constant along the dynamics \eqref{eulerODE}. Hence, without loss of generality, we may also assume it to be unity:
\begin{align}\label{rhoGamma}
	\rho_a\Gamma_a=1, \quad a \in [N].
\end{align}
Now, we substitute \eqref{rhoGamma} into other equations of \eqref{eulerODE}, and we introduce the dynamics of position variable $\mathbf{x}_a$ to the following system of relativistic TCS model:
\begin{align}\label{TCSmodel}
	\begin{aligned}
		&\frac{\mathrm{d}\mathbf{x}_a}{\mathrm{d} t}=\mathbf{v}_a,\quad a \in [N],\\
		&\frac{\mathrm{d}}{\mathrm{d} t}\left(\Gamma_a\mathbf{v}_a\left(1+\frac{p_a\Gamma_a+c_{V_a}T_a}{c^2}\right)\right)=\mathbf{M}_a,\\
		&\frac{\mathrm{d}}{\mathrm{d} t}\left[\big(p_a\left(\Gamma_a+1\right)+c^2\big)\left(\Gamma_a-1\right)+c_{V_a}\Gamma_aT_a
		\right]=\mathcal{E}_a.
	\end{aligned}
\end{align}
%%%%%%%%%%%%%
From the entropy principle, the ansetz for the production terms $\mathbf{M}_a$ and $\mathcal{E}_a$ can be made. 
In fact if we   define the following relativistic diffusion four-vectors:
\begin{align}\label{diffvect}
	W_a^{\alpha}=U_a^{\alpha}-\Gamma_aU^{\alpha}=\Gamma_a\Big((1-\Gamma)c, u_a^i\Big),
\end{align}
where
\begin{align}\label{uvGamma}
	u_a^i=v_a^i-\Gamma v^i,
\end{align}
and  we also define
\begin{align}\label{mmathE}
	\widehat{\mathbf{M}}_a=\mathbf{M}_a\equiv\left(M_a^i\right) \quad \mbox{and} \quad  \widehat{\mathcal{E}}_b=\mathcal{E}_b-\Gamma v_i\widehat{M}_a^i,
\end{align}
where $v_i=-v^i$ by the covariant-contravariant relation. In \cite{Ha-Kim-Ruggeri-ARMA-2020}, the following entropy production was evaluated:
\begin{align}\label{SigmaEm}
	\Sigma=\sum_{a=1}^{N-1}\left[\left(\frac{\Gamma_a}{T_a}-\frac{\Gamma_N}{T_N}\right)\widehat{\mathcal{E}}_a
	-\left(\frac{\Gamma_a u_{ai}}{T_a}-\frac{\Gamma_N u_{Ni}}{T_N}\right)\widehat{m}^i_a\right],
\end{align}
was obtained and therefore it satisfies the entropy inequality $\Sigma>0$ requiring as usual in thermodynamics that \eqref{SigmaEm}  is a quadratic form. This implies that there exist positive definite matrices $(\psi(|\mathbf{x}_a-\mathbf{x}_b|))$ and $(\theta(|\mathbf{x}_a-\mathbf{x}_b|))$ such that
\begin{align} 
	\begin{aligned} \label{hatmE}
		& \widehat{\mathbf{M}}_a =\frac{1}{N}\sum_{b=1}^{N-1} \psi(|\mathbf{x}_a-\mathbf{x}_b|)\Big(\frac{\Gamma_N \mathbf{u}_{N}}{T_N}-\frac{\Gamma_b \mathbf{u}_{b}}{T_b}\Big), \\
		&\widehat{\mathcal{E}}_a=\frac{1}{N}\sum_{b=1}^{N-1}\theta(|\mathbf{x}_a-\mathbf{x}_b|)\Big(\frac{\Gamma_b}{T_b}-\frac{\Gamma_N}{T_N}\Big).
	\end{aligned}
\end{align}
%%%%%%%%%%%%%%
Correspondingly, the global quantities defined in \eqref{globdens} satisfy the following set of relations:
\begin{align}\label{Sglobdens}
	\begin{aligned}
		&\rho\Gamma=N, \quad \Gamma\mathbf{v}\left(1+\frac{p\Gamma+c_{V}T}{c^2}\right)
		=\sum_{a=1}^N\Gamma_a\mathbf{v}_a\left(1+\frac{p_a\Gamma_a+c_{V_a}T_a}{c^2}\right)=\mbox{constant},\\
		&\left(p\left(\Gamma+1\right)+c^2\right)\left(\Gamma-1\right)+\Gamma c_{V}T \\
		& \hspace{2cm} =\sum_{a=1}^N\left[\left(p_a\left(\Gamma_a+1\right)+c^2\right)\left(\Gamma_a-1\right)+\Gamma_ac_{V_a}T_a
		\right]=\mbox{constant},
	\end{aligned}
\end{align}
where $c_V=\sum_{a=1}^N c_{Va}$.
We further require that the initial data of \eqref{TCSmodel} satisfy
$$\mathbf{v}(0)=\mathbf{0},\qquad T(0)=T_0.$$
Namely, we have
\begin{align}\label{initialODE}
	\begin{aligned}
		&\sum_{a=1}^N\left(\Gamma_a(0)\mathbf{v}_a(0)\left(1+\frac{p_a(0)\Gamma_a(0)+c_{V_a}T_a(0)}{c^2}\right)\right)=0,\\
		& \sum_{a=1}^N\left[\left(p_a(0)\left(\Gamma_a(0)+1\right)+c^2\right)\left(\Gamma_a(0)-1\right)+\Gamma_a(0)c_{V_a}T_a(0)
		\right]=c_{V}T_0.
	\end{aligned}
\end{align}
We combine \eqref{Sglobdens} and \eqref{initialODE} to get
$$\mathbf{v}(t)=\mathbf{0},\qquad T(t)=T_0.$$
This corresponds to the choice of special reference frame with $\mathbf{v}=0$.
From $\mathbf{v}=0$, \eqref{uvGamma}, \eqref{mmathE}, and \eqref{hatmE}, we can obtain 
\begin{align*}
	\begin{aligned} %\label{bfumE}
		& \mathbf{u}_a=\mathbf{v}_a,\quad \mathbf{M}_a=\frac{1}{N}\sum_{b=1}^{N-1}
		\psi(|\mathbf{x}_a-\mathbf{x}_b|)\left(\frac{\Gamma_N \mathbf{v}_{N}}{T_N}-\frac{\Gamma_b \mathbf{v}_{b}}{T_b}\right), \\
		& \mathcal{E}_a=\frac{1}{N}\sum_{b=1}^{N-1}\theta(|\mathbf{x}_a-\mathbf{x}_b|)\left(\frac{\Gamma_b}{T_b}-\frac{\Gamma_N}{T_N}\right).
	\end{aligned}
\end{align*}

As in \cite{Ha-Kim-Ruggeri-ARMA-2020, Ha-Ruggeri-ARMA-2017}, we transform $(N-1)\times(N-1)$ matrices $(\psi(|\mathbf{x}_a-\mathbf{x}_b|))$ and $(\theta(|\mathbf{x}_a-\mathbf{x}_b|))$ to $N\times N$ matrices $(\phi(|\mathbf{x}_a-\mathbf{x}_b|))$ and $(\zeta(|\mathbf{x}_a-\mathbf{x}_b|))$ by defining $(\phi(|\mathbf{x}_a-\mathbf{x}_b|))$ as follows:
\begin{align*}%\label{ptpz}
	\begin{aligned}
		&\phi(|\mathbf{x}_i-\mathbf{x}_j|):=-\psi(|\mathbf{x}_i-\mathbf{x}_j|), \quad \forall~ i\neq j\in [N-1],\\
		&\phi(|\mathbf{x}_i-\mathbf{x}_N|):=\phi(|\mathbf{x}_N-\mathbf{x}_i|)=\sum_{j=1}^{N-1}\psi(|\mathbf{x}_i-\mathbf{x}_j|), \quad \forall~ i \in [N-1],\\
		&\phi(|\mathbf{x}_a-\mathbf{x}_a|):=\phi(0), \quad \forall~a \in [N].
	\end{aligned}
\end{align*}
We also define $(\zeta(|\mathbf{x}_a-\mathbf{x}_b|))$ in a similar manner. After we use the change of matrices, system \eqref{TCSmodel} becomes
\begin{align}\label{TCSmodel-1}
	\begin{aligned}
		&\frac{\mathrm{d}\mathbf{x}_a}{\mathrm{d} t}=\mathbf{v}_a,\quad a \in [N], \\
		&\frac{\mathrm{d}}{\mathrm{d} t}\left(\Gamma_a\mathbf{v}_a\left(1+\frac{p_a\Gamma_a+c_{V_a}T_a}{c^2}\right)\right)
		=\frac{1}{N}\sum_{b=1}^{N}\phi(|\mathbf{x}_a-\mathbf{x}_b|)\Big(\frac{\Gamma_b \mathbf{v}_{b}}{T_b}-\frac{\Gamma_a \mathbf{v}_{a}}{T_a}\Big),\\
		&\frac{\mathrm{d}}{\mathrm{d} t}\left[\left(p_a\left(\Gamma_a+1\right)+c^2\right)\left(\Gamma_a-1\right)+\Gamma_ac_{V_a}T_a
		\right]=\frac{1}{N}\sum_{b=1}^{N}\zeta(|\mathbf{x}_a-\mathbf{x}_b|)\Big(\frac{\Gamma_a }{T_a}-\frac{\Gamma_b }{T_b}\Big),
	\end{aligned}
\end{align}
subject to the initial data with the conditions \eqref{initialODE}. Note that
\begin{align*}
	\begin{aligned}
		&\lim_{c\rightarrow \infty} \Gamma_a=1,\qquad \lim_{c\rightarrow \infty} c^2(\Gamma_a-1)=\frac{v_a^2}{2}.
	\end{aligned}
\end{align*}
Then, the relations \eqref{TCSmodel-1} reduce to the classical TCS model proposed by Ha and Ruggeri \cite{Ha-Ruggeri-ARMA-2017}:
\begin{equation*}%\label{CTCSmodel-1}
	\begin{cases}
		&\displaystyle\frac{\mathrm{d}\mathbf{x}_a}{\mathrm{d} t}=\mathbf{v}_a,\quad a \in [N], \\
		&\displaystyle\frac{\mathrm{d} \mathbf{v}_a}{\mathrm{d} t}
		=\frac{1}{N}\sum_{b=1}^{N}\phi(|\mathbf{x}_a-\mathbf{x}_b|)\Big(\frac{ \mathbf{v}_{b}}{T_b}-\frac{\mathbf{v}_{a}}{T_a}\Big),\\
		&\displaystyle\frac{\mathrm{d}}{\mathrm{d} t}\left(\frac{v_a^2}{2}+c_{V_a}T_a\right)=\frac{1}{N}\sum_{b=1}^{N}\zeta(|\mathbf{x}_a-\mathbf{x}_b|)\Big(\frac{1 }{T_a}-\frac{1 }{T_b}\Big),
	\end{cases}
\end{equation*}
subject to the reduced form of initial data constrained by \eqref{initialODE}:
\begin{align*}
	\begin{aligned}
		&\sum_{a=1}^N\mathbf{v}_a(0)=0,\qquad \sum_{a=1}^N\left(\frac{v_a^2(0)}{2}+c_{V_a}T_a(0)\right)=c_{V}T_0.
	\end{aligned}
\end{align*}
\subsubsection{Mechanical ensemble} \label{sec:2.4}
In this subsection, we recall the RCS model from the RTCS model obtained in \cite{Ha-Kim-Ruggeri-ARMA-2020}.
According with the theory of principal sub-systems \cite{B-R} we need to take constant the \emph{multiplier} of the energy equation and this implies :
\begin{align}\label{TTa}
	\frac{T_a}{\Gamma_a}=T^*, \quad a \in [N].
\end{align}
where $T^*$ is constant.
Then we substitute \eqref{TTa} into \eqref{TCSmodel-1}, omit the last equation,  and we recall that 
\[ c_{V_a}=\frac{D_a}{2}\frac{k_B }{m_a}=\frac{D_a c^2}{2T^*\gamma^*_a} \]
with the dimensionless quantity  $D_a$ represents the degree of freedom of the gas  and we have put as usual in relativistic fluid the dimensionless constant
\[
\gamma^*_a = \frac{m_a c^2}{k_B T^*}.
\]
Then the system read
\begin{equation}\label{CSmodel}
	\begin{cases}
		&\displaystyle\frac{\mathrm{d}\mathbf{x}_a}{\mathrm{d} t}=\mathbf{v}_a,\qquad a \in [N], \\
		&\displaystyle\frac{\mathrm{d}}{\mathrm{d} t}\left(\Gamma_a\mathbf{v}_a\left(1+\frac{D_a+2}{2\gamma^*_a}\Gamma_a\right)\right)
		=\frac{1}{NT^*}\sum_{b=1}^{N}\phi(|\mathbf{x}_a-\mathbf{x}_b|)\left(\mathbf{v}_{b}-\mathbf{v}_{a}\right),
	\end{cases}
\end{equation}
subject to the initial data with the following condition
\begin{align*}%\label{initialODE-sub}
	\begin{aligned}
		&\sum_{a=1}^N\left(\Gamma_a(0)\mathbf{v}_a(0)\left(1+\frac{D_a+2}{2\gamma^*_a}\Gamma_a(0)\right)\right)=0.
	\end{aligned}
\end{align*}
Corresponding to \eqref{CSmodel}, as $c\rightarrow \infty$, the classical CS model reads:
\begin{equation}\label{CCSmodel}
	\begin{cases}
		&\displaystyle\frac{\mathrm{d}\mathbf{x}_a}{\mathrm{d} t}=\mathbf{v}_a,\quad a \in [N],\\
		&\displaystyle\frac{\mathrm{d}}{\mathrm{d} t}\mathbf{v}_a
		=\frac{1}{NT^*}\sum_{b=1}^{N}\phi(|\mathbf{x}_a-\mathbf{x}_b|)\left(\mathbf{v}_{b}-\mathbf{v}_{a}\right),
	\end{cases}
\end{equation}
subject to the initial data with the following condition:
\begin{align*}%\label{initialCODE-sub}
	\begin{aligned}
		&\sum_{a=1}^N\mathbf{v}_a(0)=0.
	\end{aligned}
\end{align*}
In what follows, the communication weight function $\phi:\mathbb{R}_+\rightarrow\mathbb{R}_+$ is nonnegative, bounded Lipschitz continuous, and monotonically decreasing,
\begin{align*}%\label{Lipsch-phi}
	\begin{aligned}
		&0\leq \phi(r)\leq \phi(0)=1,\qquad \phi(\cdot)\in \mathrm{Lip}(\mathbb{R}_+;\mathbb{R}_+),\\
		&(\phi(r_1)-\phi(r_2))(r_1-r_2)\leq0,\qquad \forall~ r, r_1, r_2\geq0.
	\end{aligned}
\end{align*}

\subsection{Main results} % \label{sec：2.2}
In this subsection, we briefly summaries our main results. 
To explain the aim of this paper, we first recall the concept of asymptotic flocking for the RCS model and the CS model. For this, we introduce two Lipschitz continuous functionals:~for time-dependent configuration $\{ ({\bf x}_a, {\bf w}_a) \}$: we set
\[
{\mathcal D}_{\bf x}:=  \max_{a, b \in [N]} |{\bf x}_{a}-{\bf x}_{b}|, \quad {\mathcal D}_{\bf w}:=  \max_{a, b \in [N]} |{\bf w}_{a}-{\bf w}_{b}|.
\]
Here and below, as in \cite{Ha-Kim-Ruggeri-ARMA-2020}, we introduce auxiliary observables:
\[
\mathbf{w}_a:=\Gamma_a\mathbf{v}_a\left(1+\frac{D_a+2}{2\gamma^*_a}\Gamma_a\right), \quad F_a:=\Gamma_a\left(1+\frac{D_a+2}{2\gamma^*_a}\Gamma_a\right),\quad a \in [N].
\]
\begin{definition} Let $\{({\bf x}_{a}, {\bf w}_{a})\}$ be a global solution to the RCS model. Then an asymptotic flocking emerges if and only if the following estimates hold:
\[
\sup_{0\leq t<\infty} {\mathcal D}_{\bf x}(t) <\infty,\quad \lim_{t \to \infty} {\mathcal D}_{\bf w}(t) = 0. 
\]
\end{definition}
In this paper, we revisit the following questions which have been discussed in earlier works:
\begin{itemize}
\item
(Q1):~Can we provide conditions for asymptotic flocking in terms of system functions and initial data?
\vspace{0.2cm}
\item
(Q2):~What is an optimal convergence rate in the uniform-time classical limit of the RCS models in terms of the speed of light?
\end{itemize}
Of course, the above two questions have been addressed in previous works \cite{Ahn-Ha-Kim-JMP-2022,Ahn-Ha-Kim-CPAA-2021,Ha-Kim-Ruggeri-ARMA-2020, Ha-Kim-Ruggeri-CMS-2021} with rather crude conditions and indirect arguments (see Remark \ref{R3.2} and Remark \ref{R3.3}).  \newline

The main results of this paper are four-fold: First, we present a sufficient framework leading to the asymptotic flocking of the RCS model \eqref{CSmodel} under a refined condition on initial data and system functions:
 \[
   \mathcal{D}_\mathbf{x}(0)+\frac{c^4T^*\mathcal{D}_{\mathbf{w}}(0)}{(c^2+1)[2c^2\phi(\mathcal{D}^{\infty}_{\mathbf{x}})-4\Lambda_2(\mathcal{D}_{\mathbf{w}}(0))]}
      <\mathcal{D}_\mathbf{x}^{\infty},
 \]
 for a positive constant $\mathcal{D}_\mathbf{x}^{\infty},$ and  let $\{(\mathbf{x}_{a}, \mathbf{w}_{a})\}$ be a global solution to \eqref{CSmodel-w}.  Then, there exists a positive constant $\lambda = \lambda( \mathcal{D}_\mathbf{w}(0), \mathcal{D}_{\mathbf{x}}^{\infty}, T^*)$  such that 
 \[ \sup_{0\leq t<\infty}\mathcal{D}_{\mathbf{x}}(t)<\mathcal{D}_{\mathbf{x}}^{\infty}, \quad  \mathcal{D}_\mathbf{w}(t)\leq \mathcal{D}_\mathbf{w}(0) \exp\left\{-\lambda t\right\}, \quad t > 0.
 \]
We refer to Theorem \ref{T3.1} for details and Remark \ref{R3.2} for the comparison with results in previous works. 

Second, we provide an optimal convergence speed in the uniform-time classical limit from the RCS model \eqref{CSmodel} to the CS model \eqref{CCSmodel} as $c \to \infty$. For a finite-time classical limit, the convergence speed can be shown at the order of ${\mathcal O}(c^{-4})$. However, the argument employed in the uniform-time classical limit \cite{Ha-Kim-Ruggeri-ARMA-2020} was indirect without any explicit convergence rate. In this work, we use a direct argument to derive an optimal convergence rate ${\mathcal O}(c^{-4})$ which is consistent with finite-time classical limit. To quantify the classical limit, we introduce a functional:
\[ \Delta^c(t):=\frac{1}{N}\sum_{a=1}^N\left(\left|\mathbf{x}_a(t)-\mathbf{x}^{\infty}_a(t)\right|^2
+\left|\mathbf{w}_a(t)-\mathbf{w}^{\infty}_a(t)\right|^2\right). \]
Under a suitable framework in terms of system parameters and initial data, we show that if $\Delta^c(0)$ satisfies
\[\Delta^c(0)\leq \mathcal{O}(c^{-4}),\]
which is weaker than the corresponding assumption $\Delta^c(0)=0$ in  \cite{Ha-Kim-Ruggeri-ARMA-2020},
then we have
\[
\sup_{0 \leq t < \infty} \Delta^c(t)\leq \mathcal{O}(c^{-4}).
\]
We refer to Theorem \ref{T3.2} and Remark \ref{R3.3} for details. 

Third, we consider flocking estimate for the kinetic RCS model \eqref{CSmodel} which can be obtained from the RCS model via the mean-field limit \cite{AH}. 
\begin{equation} \label{A-3}
\begin{cases}
\displaystyle \partial_t f+\frac{{\bf w}}{F}\cdot \nabla_{\bf x}f+\nabla_{\bf w}\cdot\left(L[f]f\right)=0,\qquad (t,{\bf x},{\bf w})\in \mathbb{R}_+\times\mathbb{R}^6, \\
\displaystyle L[f](t,{\bf x},{\bf w}):=-\int_{\mathbb{R}^6}\phi(|{\bf x}_*-{\bf x}|)\left(\frac{{\bf w}}{F}-\frac{{\bf w}_*}{F_*}\right)f(t, {\bf x}_*, {\bf w}_*)\, \mathrm{d} {\bf x}_*\mathrm{d} {\bf w}_*, \\
\displaystyle f \Big|_{t=0_+}=f^{\mathrm{in}},
\end{cases}
\end{equation} 
where $f\equiv f(t,{\bf x},{\bf w})$ is the one-particle distribution function, ${\bf w}/F$ is the microscopic  velocity and 
$$F\equiv F (w^2)=\left(1+\frac{D+2}{2\gamma^*}\Gamma\right)\Gamma.$$ 

Suppose that $f$ tends to $f^{\infty}$ in a suitable weak sense as $c\rightarrow \infty$. Then, the limit $f^{\infty}$ satisfies the kinetic CS model:
\begin{equation}\label{A-4}
\begin{cases}
\displaystyle \partial_t f^{\infty}+{\bf w}\cdot \nabla_{\bf x}f^{\infty}+\nabla_{\bf w}\cdot\left(L[f^{\infty}]f^{\infty}\right)=0,\qquad (t,{\bf x},{\bf w})\in \mathbb{R}_+\times\mathbb{R}^6, \\
\displaystyle L[f^{\infty}](t,{\bf x},{\bf w}):=-\int_{\mathbb{R}^6}\phi(|{\bf x}_*-{\bf x}|)({\bf w}-{\bf w}_*)f^{\infty}(t, {\bf x}_*, {\bf w}_*)\, \mathrm{d} {\bf x}_*\mathrm{d} {\bf w}_*, \\
\displaystyle f^{\infty} \Big |_{t=0_+}=f^{\infty,\mathrm{in}}.
\end{cases}
\end{equation}
Next, we briefly discuss a sufficient framework for the men-field limit for the kinetic RCS model. Suppose that the initial datum $f^{\mathrm{in}}$ is compactly supported and satisfies 
\begin{align}\label{inDxw}
\begin{aligned}
   &\int_{\mathbb{R}^6}f^{\mathrm{in}}\, \mathrm{d}{\bf  x}\mathrm{d} {\bf w}=1,\quad\int_{\mathbb{R}^6}{\bf w}f^{\mathrm{in}}\, \mathrm{d}{\bf  x}\mathrm{d} {\bf w}=0, \\
   & \mathcal{D}_{\bf x}(0)+\frac{c^4\mathcal{D}_{\bf w}(0)}{(c^2+1)[c^2\phi(\mathcal{D}_{\bf x}^{\mathrm{\infty}})-2\Lambda_2(D_\mathbf{w}(0))]}
      <\mathcal{D}_{\bf x}^{\mathrm{\infty}}.
      \end{aligned}
 \end{align} 
 Under the improved conditions \eqref{inDxw}, we have
 \[ \sup_{0\leq t<\infty}\mathcal{D}_{\bf x}(t)<\mathcal{D}_{\bf x}^{\infty},\qquad \mathcal{D}_{\bf w}(t)\leq \mathcal{D}_{\bf w}(0) \exp\left\{-\left(\phi(\mathcal{D}_{\bf x}^{\infty})-\frac{2\Lambda_2(D_\mathbf{w}(0))}{c^2}\right) t\right\}. \]
 See Theorem \ref{T4.1} and Remark \ref{R4.1} for details. \newline
 
Fourth, we consider the uniform-time classical limit from  kinetic RCS model \eqref{A-3} to kinetic CS model \eqref{A-4} as $c \to \infty$. In this case, we suppose that the initial probability Radon measure $\mu_0$ satisfies the conditions in Proposition \ref{KCS-lim}, and  let $\mu_t$ and $\mu_t^{\infty}$ be measure-valued solutions to Cauchy problems to \eqref{A-3} and \eqref{A-4}  with the compactly supported initial measures $\mu_0$ and $\mu_0^{\infty}$ satisfying
$$    W_1(\mu_0, \mu^{\infty}_0) \leq \mathcal{O}(c^{-2}).$$
    Then, we derive the following quantitative uniform-time classical limit: 
\[
\sup_{0 \leq t < \infty} W_1(\mu_t, \mu^{\infty}_t) \leq \mathcal{O}(c^{-2}),
\]
where $W_1(\cdot, \cdot)$ is the 1-Wasserstein metric.  We refer to Theorem \ref{T4.2} and Remark \ref{R4.2} for details.

\section{Qualitative analysis for RCS model} \label{sec:3}
\setcounter{equation}{0}
In this section, we derive several quantitative estimates on flocking dynamics and uniform-time  classical limit ($c\rightarrow \infty$) from the RCS model \eqref{CSmodel} to the classical CS model. Via a novel direct proof, we obtain classical limit with an optimal convergence rate.

%%%%%%%%%%%%%%%%%%%%%
%
%
%%%%%%%%%%%%%%%%%%%%%%%
\subsection{Asymptotic flocking dynamics}\label{sec:3.1}
In this subsection, we study asymptotic flocking estimate. 
We study an emergent flocking dynamics to the RCS model \eqref{CSmodel}.

System \eqref{CSmodel} and initial condition can be written as
\begin{equation}
\begin{cases} \label{CSmodel-w}
\displaystyle \frac{\mathrm{d}\mathbf{x}_a}{\mathrm{d} t}=\frac{\mathbf{w}_a}{F_a},\quad a \in [N],\\
\displaystyle \frac{\mathrm{d}\mathbf{w}_a}{\mathrm{d} t}
=\frac{1}{NT^*}\sum_{b=1}^{N}\phi(|\mathbf{x}_a-\mathbf{x}_b|)\left(\frac{\mathbf{w}_b}{F_b}-\frac{\mathbf{w}_a}{F_a}\right), \\
\displaystyle \sum_{a=1}^N\mathbf{w}_a(0)=0. 
\end{cases}
\end{equation}
Note that the zero sum condition $\eqref{CSmodel-w}_3$ will be propagated along $\eqref{CSmodel-w}_3$:
\begin{align}\label{initialODE-sub-wt}
\begin{aligned}
&\sum_{a=1}^N\mathbf{w}_a(t)=0,\qquad \forall~t\geq0. 
\end{aligned}
\end{align}
%Since the classical limit of $\mathbf{w}_a$ is $\mathbf{v}_a$, there exists a constant $W_0$ independent of c such that
%\begin{align}\label{uni-con}
%\sup_{0\leq t<\infty}\max_{1\leq a\leq N}\left|\mathbf{w}_a\right|\leq W_0
%\end{align}
%by the uniform upper bound of $\mathbf{v}_a$ in Lemma \ref{bound-va}. 

In what follows, we study several preparatory estimates for asymptotic flocking estimates. Note that
\begin{align*}
&\Gamma:=\frac{1}{\sqrt{1-\frac{v^2}{c^2}}}, \quad \mathbf{w}:=\mathbf{v}\left(1+\frac{D+2}{2\gamma^*}\Gamma\right)\Gamma,\quad v=|\mathbf{v}|,\quad w=|\mathbf{w}|,\\
&v^2\Gamma^2=c^2(\Gamma^2-1),\quad w^2=c^2(\Gamma^2-1)\left(1+\frac{D+2}{2\gamma^*}\Gamma\right)^2.
\end{align*}
We set 
\begin{align*}
F(w^2) :=\left(1+\frac{D+2}{2\gamma^*}\Gamma\right)\Gamma.
\end{align*}
Before we move on to quantitative estimates on flocking dynamics, we provide a new proof of an important estimate corresponding to that in \cite[Lemma 6.6]{Ha-Kim-Ruggeri-ARMA-2020}. 
\begin{lemma}\label{coer-F}
Let $\mathbf{x}$ and $\mathbf{y}$ be two vectors in $\mathbb{R}^3$ such that 
\begin{align*}
\left|\mathbf{x}\right|\leq W_0,\quad \left|\mathbf{y}\right|\leq W_0.
\end{align*}
Then, there exists a positive constant $\Lambda_0=\Lambda_0(W_0)$ depending on R such that
\begin{align}\label{coer-xy}
\left(\mathbf{x}-\mathbf{y}\right)\cdot\left(\frac{\mathbf{x}}{F(x^2)}-\frac{\mathbf{y}}{F(y^2)}\right)\geq \Lambda_0(W_0)\left|\mathbf{x}-\mathbf{y}\right|^2.
\end{align}
\end{lemma}
\begin{proof} For $\mathbf{z}=(z_1,z_2,z_3)\in \mathbb{R}^3$ such that 
\[ z^2=\left|\mathbf{z}\right|^2\leq W_0^2, \]
we denote the matrix $A=A(\mathbf{z})$ as
 $$A(\mathbf{z}):=\nabla_\mathbf{z} \left(\frac{\mathbf{z}}{F(z^2)}\right)= \begin{pmatrix}
\frac{1}{F(z^2)}-\frac{2z_1^2F'(z^2)}{F^2(z^2)} & -\frac{2z_1z_2F'(z^2)}{F^2(z^2)}& -\frac{2z_1z_3F'(z^2)}{F^2(z^2)}\\
-\frac{2z_1z_2F'(z^2)}{F^2(z^2)} & \frac{1}{F(z^2)}-\frac{2z_2^2F'(z^2)}{F^2(z^2)}  & -\frac{2z_2z_3F'(z^2)}{F^2(z^2)} \\
-\frac{2z_1z_3F'(z^2)}{F^2(z^2)} & -\frac{2z_2z_3F'(z^2)}{F^2(z^2)} & \frac{1}{F(z^2)}-\frac{2z_3^2F'(z^2)}{F^2(z^2)} 
\end{pmatrix}. $$
By simple computations, we have the following eigenvalues of $A(z)$:
$$\lambda_3=\lambda_2=\frac{1}{F(z^2)}>\lambda_1= \frac{1}{F(z^2)}-\frac{2z^2F'(z^2)}{F^2(z^2)}.$$                                                         
Then, in order to verify \eqref{coer-xy}, it suffices to show that there exists a positive constant $\Lambda_0(W_0)$ such that 
$$\frac{1}{F(z^2)}-\frac{2z^2F'(z^2)}{F^2(z^2)}=\frac{F(z^2)-2z^2F'(z^2)}{F^2(z^2)}\geq \Lambda_0(W_0)>0,\quad z\leq W_0.$$ 
We first need to show 
$$F(z^2)-2z^2F'(z^2)>0, \quad z\leq W_0. $$
Note that 
\[ F'(z^2)=\left(1+\frac{D+2}{\gamma^*}\Gamma\right)\partial_{z^2}\Gamma, \quad 1=2c^2\left(1+\frac{D+2}{2\gamma^*}\Gamma\right)\left(\frac{D+2}{\gamma^*}\Gamma^2+\Gamma-\frac{D+2}{2\gamma^*}\right)\partial_{z^2}\Gamma.
\]
Then we have
\begin{align*}
   &F'(z^2)=\frac{1+\frac{D+2}{\gamma^*}\Gamma}
   {2c^2\left(1+\frac{D+2}{2\gamma^*}\Gamma\right)\left(\frac{D+2}{\gamma^*}\Gamma^2+\Gamma-\frac{D+2}{2\gamma^*}\right)}.
\end{align*}
For $z^2=c^2(\Gamma^2-1)\left(1+\frac{D+2}{2\gamma^*}\Gamma\right)^2\leq W_0^2$, we have
\begin{align*}
&F(z^2)-2z^2F'(z^2)\\
&\quad=\left(1+\frac{D+2}{2\gamma^*}\Gamma\right)\Gamma-
\frac{(\Gamma^2-1)\left(1+\frac{D+2}{2\gamma^*}\Gamma\right)\left(1+\frac{D+2}{\gamma^*}\Gamma\right)}
   {\frac{D+2}{\gamma^*}\Gamma^2+\Gamma-\frac{D+2}{2\gamma^*}}\\
   &\quad> \frac{\Gamma\left(\frac{D+2}{\gamma^*}\Gamma^2+\Gamma-\frac{D+2}{2\gamma^*}\right)-(\Gamma^2-1)\left(1+\frac{D+2}{\gamma^*}\Gamma\right)}
   {\frac{D+2}{\gamma^*}\Gamma^2+\Gamma-\frac{D+2}{2\gamma^*}}\\
   &\quad>\frac{1}
   {   \frac{D+2}{\gamma^*}\Gamma^2+\Gamma-\frac{D+2}{2\gamma^*}}>0.
\end{align*}
Note that 
$$\lim_{c\rightarrow\infty}F(z^2)=\lim_{c\rightarrow\infty}\left(1+\frac{D+2}{2\gamma^*}\Gamma\right)\Gamma=1, \qquad \lim_{c\rightarrow\infty}\frac{1}
   {   \frac{D+2}{\gamma^*}\Gamma^2+\Gamma-\frac{D+2}{2\gamma^*}}=1. $$
Then, there exists a positive constant $\Lambda_0(W_0)$ independent of $c$ such that the relation \eqref{coer-xy} holds.
\end{proof}
\begin{remark}
In the proof of Lemma \ref{coer-F}, one can see 
$$\Lambda_0(W_0)\leq 1,\qquad \lim_{c\rightarrow\infty}\Lambda_0(W_0)=1.$$
\end{remark}

Note that diameter functionals $\mathcal{D}_\mathbf{x}$ and $\mathcal{D}_\mathbf{w}$ are Lipschitz continuous and hence they are almost everywhere differentiable.  In the following lemma, we show that diameters $\mathcal{D}_\mathbf{x}$ and $\mathcal{D}_\mathbf{w}$ satisfy the ``system of dissipative differential inequalities (SDDI)".

\begin{lemma}\label{lemma-SDDI} 
Let $\{(\mathbf{x}_{a}, \mathbf{w}_{a})\}$ be a global solution to \eqref{CSmodel-w} with compactly supported initial data. Then, the diameters $(\mathcal{D}_\mathbf{x}, \mathcal{D}_\mathbf{w})$ satisfy 
\begin{equation}
\begin{cases} \label{SDDI-r}
\displaystyle \frac{\mathrm{d}\mathcal{D}_\mathbf{x}}{\mathrm{d} t}\leq\frac{c^2}{c^2+1}\mathcal{D}_\mathbf{w}, \vspace{0.2cm}\\
\displaystyle \frac{\mathrm{d}\mathcal{D}_\mathbf{w}}{\mathrm{d} t}\leq -\left(\frac{\phi(\mathcal{D}_{\mathbf{x}})}{T^*}-\frac{2\Lambda_2(\mathcal{D}_{\mathbf{w}})}{c^{2}T^*}\right)\mathcal{D}_\mathbf{w}.
\end{cases}
\end{equation}
\end{lemma}
\begin{proof} For given $t \in (0, \infty)$, we choose time-dependent extremal pairs of indices $(M_t, m_t)$ such that 
\begin{equation} \label{NN-1}
\mathcal{D}_\mathbf{x}(t):=\left|\mathbf{x}_{M_t}(t)-\mathbf{x}_{m_t}(t)\right|.
\end{equation}

From now on, as long as there is no confusion, we suppress $t$ dependence in the above extremal indices. \newline

\noindent  $\bullet$ (Derivation of $\eqref{SDDI-r}_1$):~It follows from the definition of $\mathcal{D}_\mathbf{x}$ that 
  \begin{equation*}
    \frac{\mathrm{d}\mathcal{D}^2_\mathbf{x}}{\mathrm{d} t}=2\mathcal{D}_\mathbf{x}\frac{\mathrm{d}\mathcal{D}_\mathbf{x}}{\mathrm{d} t}
    =2\mathcal{D}_\mathbf{x}\frac{\mathrm{d}\left|\mathbf{x}_M-\mathbf{x}_m\right|}{\mathrm{d} t}\leq 2\mathcal{D}_\mathbf{x} \left|\mathbf{v}_{M}-\mathbf{v}_{m} \right|.
  \end{equation*}
  According to \cite[Corollary $2.1$]{Ahn-Ha-Kim-JMP-2022}, it holds that
  \begin{equation*}
     \left|\mathbf{v}_M-\mathbf{v}_m\right|\leq \frac{c^2}{c^2+1} \left|\mathbf{w}_M-\mathbf{w}_m\right| \leq \frac{c^2}{c^2+1}\mathcal{D}_\mathbf{w}.
  \end{equation*}
The above two inequalities yield the  first inequality in \eqref{SDDI-r}. \newline
  
\noindent $\bullet$  (Derivation of $\eqref{SDDI-r}_2$):~ Similar to \eqref{NN-1}, we choose extremal indices $M$ and $m$ such that
 $$\left|\mathbf{w}_M(t)-\mathbf{w}_m(t)\right|=\max_{1\leq a,b\leq N}\left|\mathbf{w}_a(t)-\mathbf{w}_b(t)\right|.$$ 
 Although this pair $(M, m)$ might be different from \eqref{NN-1}, but we still use the same notation. Then we use $\eqref{CSmodel-w}_2$ to see
 \begin{align*}
 \begin{aligned}
\mathcal{D}_\mathbf{w}\frac{\mathrm{d}\mathcal{D}_\mathbf{w}}{\mathrm{d} t} &=\frac{1}{2}\frac{\mathrm{d}\mathcal{D}_\mathbf{w}^2}{\mathrm{d} t}=\left(\mathbf{w}_M-\mathbf{w}_m\right)\cdot\frac{\mathrm{d}}{\mathrm{d} t}\left(\mathbf{w}_M-\mathbf{w}_m\right) \\
&=\left(\mathbf{w}_M-\mathbf{w}_m\right)\cdot
\frac{1}{NT^*} \\
& \times \sum_{b=1}^{N}\left[\phi(|\mathbf{x}_M-\mathbf{x}_b|)\left(\frac{\mathbf{w}_b}{F_b}-\frac{\mathbf{w}_M}{F_M}\right)
-\phi(|\mathbf{x}_m-\mathbf{x}_b|)\left(\frac{\mathbf{w}_b}{F_b}-\frac{\mathbf{w}_m}{F_m}\right)\right ]  \\
&=
\frac{\left(\mathbf{w}_M-\mathbf{w}_m\right)\cdot}{NT^*}\sum_{b=1}^{N}\Big[\phi(|\mathbf{x}_M-\mathbf{x}_b|) \left(\mathbf{w}_b-\mathbf{w}_M\right)
-\phi(|\mathbf{x}_m-\mathbf{x}_b|)
\left(\mathbf{w}_b-\mathbf{w}_m\right)\\
& +\phi(|\mathbf{x}_M-\mathbf{x}_b|)\big[\big(\frac{1}{F_b}-1\big)\mathbf{w}_b
-\big(\frac{1}{F_M}-1\big)\mathbf{w}_M\big]\\
&-\phi(|\mathbf{x}_m-\mathbf{x}_b|)\big[\big(\frac{1}{F_b}-1\big)\mathbf{w}_b
-\big(\frac{1}{F_M}-1\big)\mathbf{w}_m\big]\Big]\\
&\leq -\frac{\phi(\mathcal{D}_{\mathbf{x}})}{T^*}\mathcal{D}_\mathbf{w}^2+\frac{D_\mathbf{w}}{NT^*} 
\sum_{b=1}^{N}\left(\left|\nabla_{\mathbf{z}_1}\left[\left(\frac{1}{F(z^2_1)}-1\right)\mathbf{z}_1\right]
\right|\left|\mathbf{w}_M-\mathbf{w}_b\right|\right)\\
&+\frac{D_\mathbf{w}}{NT^*} 
\sum_{b=1}^{N}\left(\left|\nabla_{\mathbf{z}_2}\left[\left(\frac{1}{F(z^2_2)}-1\right)\mathbf{z}_2\right]
\right|\left|\mathbf{w}_b-\mathbf{w}_m\right|\right),
\end{aligned}
\end{align*}
where we used 
$$ \left(\mathbf{w}_M-\mathbf{w}_m\right)\cdot \left(\mathbf{w}_b-\mathbf{w}_M\right)\leq 0,\quad\left(\mathbf{w}_M-\mathbf{w}_m\right)\cdot \left(\mathbf{w}_b-\mathbf{w}_m\right)\geq 0.$$ 
Here
$$\mathbf{z}_1=\tau_1 \mathbf{w}_M+(1-\tau_1)\mathbf{w}_b, \quad \mathbf{z}_2=\tau_2 \mathbf{w}_b+(1-\tau_2)\mathbf{w}_m$$ for some $\tau_1, \tau_2\in(0,1)$.
 Note that  the conservation of the relativistic momentum in \eqref{initialODE-sub-wt} implies
\begin{equation} \label{wD-relation}
|\mathbf{w}_b|=\left|\mathbf{w}_b-\frac{1}{N}\sum_{a=1}^N\mathbf{w}_a\right|\leq \mathcal{D}_\mathbf{w}, \quad \left|\mathbf{z}_i\right|\leq \mathcal{D}_{\mathbf{w}}
\end{equation}
for $i=1, 2$. Then, it follows from the proof of Lemma \ref{coer-F} that
 \begin{align} 
 \begin{aligned} \label{I21}
 &\left|\nabla_{\mathbf{z}_i}\left[\left(\frac{1}{F(z^2_i)}-1\right)\mathbf{z}_i\right]
\right|\leq  \left|\frac{1}{F(z^2_i)}-1\right|+2z^2_i\frac{F'(z^2_i)}{F^2(z^2_i)} \\
&\hspace{1cm} \leq \frac{\Gamma-1+\frac{D+2}{2\gamma^*}\Gamma^2}{\left(1+\frac{D+2}{2\gamma^*}\Gamma\right)\Gamma}
+\frac{(\Gamma^2-1)\left(1+\frac{D+2}{\gamma^*}\Gamma\right)}
   {\left(1+\frac{D+2}{2\gamma^*}\Gamma\right)
   \left(\frac{D+2}{\gamma^*}\Gamma^2+\Gamma-\frac{D+2}{2\gamma^*}\right)\Gamma^2}\\
   &\hspace{1cm} \leq \frac{\Gamma^2-1}{\Gamma+1}+\frac{(D+2)k_BT^*}{2mc^2}\Gamma^2
   +(\Gamma^2-1)\left(1+\frac{(D+2)k_BT^*}{mc^2}\Gamma\right) \\
   &\hspace{1cm} \leq \frac{\Lambda_2(D_{\mathbf{w}})}{c^{2}}. 
\end{aligned}
 \end{align}
Now we can further use \eqref{I21} to obtain
 \begin{align*}%\label{wMm}
\frac{1}{2}\frac{\mathrm{d}\mathcal{D}_\mathbf{w}^2}{\mathrm{d} t} &=\left(\mathbf{w}_M-\mathbf{w}_m\right)\cdot\frac{\mathrm{d}}{\mathrm{d} t}\left(\mathbf{w}_M-\mathbf{w}_m\right)\nonumber\\
&\leq -\frac{\phi(\mathcal{D}_{\mathbf{x}})}{T^*}\mathcal{D}_\mathbf{w}^2+\frac{\mathcal{D}_\mathbf{w}}{NT^*} 
\sum_{b=1}^{N}\frac{\Lambda_2(\mathcal{D}_{\mathbf{w}})}{c^{2}}\left(\left|\mathbf{w}_M-\mathbf{w}_b\right|
+\left|\mathbf{w}_b-\mathbf{w}_m\right|\right)\\
&\leq \left(-\frac{\phi(\mathcal{D}_{\mathbf{x}})}{T^*}-\frac{2\Lambda_2(\mathcal{D}_{\mathbf{w}})}{c^{2}T^*}\right)\mathcal{D}_\mathbf{w}^2.\nonumber
\end{align*} 
  This implies the second inequality in \eqref{SDDI-r}.
\end{proof}
\vspace{0.5cm}

Now we are ready to present an asymptotic flocking result to the RCS model \eqref{CSmodel-w}. Although the asymptotic  flocking estimate of the RCS model has already been studied in \cite{Ahn-Ha-Kim-JMP-2022, Ahn-Ha-Kim-CPAA-2021, Ha-Kim-Ruggeri-ARMA-2020},  we derive improved flocking estimate in the following theorem (see Remark  \ref{R3.2}). For a later use, we set 
\[
\lambda = \lambda( \mathcal{D}_\mathbf{w}(0), \mathcal{D}_{\mathbf{x}}^{\infty}, T^*,c) := \frac{\phi(\mathcal{D}^{\infty}_{\mathbf{x}})}{T^*}-\frac{2\Lambda_2(\mathcal{D}_{\mathbf{w}}(0))}{c^{2}T^*}. 
\]
 \begin{theorem}\label{T3.1}
 Suppose that initial data satisfy 
 \begin{equation} \label{NN-2-0}
  D_\mathbf{x}(0)+\frac{c^4T^*\mathcal{D}_{\mathbf{w}}(0)}{(c^2+1)[2c^2\phi(\mathcal{D}^{\infty}_{\mathbf{x}})-4\Lambda_2(\mathcal{D}_{\mathbf{w}}(0))]}
      <\mathcal{D}_\mathbf{x}^{\infty},
 \end{equation}
 for a positive constant $\mathcal{D}_\mathbf{x}^{\infty},$ and  let $\{(\mathbf{x}_{a}, \mathbf{w}_{a})\}$ be a global solution to \eqref{CSmodel-w}. Then, we have asymptotic flocking:
 \begin{align*}
 \begin{aligned} 
\sup_{0\leq t<\infty}\mathcal{D}_{\mathbf{x}}(t)<\mathcal{D}_{\mathbf{x}}^{\infty}, \quad \mathcal{D}_\mathbf{w}(t)\leq \mathcal{D}_\mathbf{w}(0) \exp\left\{-\lambda t\right\}, \quad t > 0.
 \end{aligned}
 \end{align*}
\end{theorem}
  \begin{proof} 
  \noindent (i)~First, note that $\eqref{SDDI-r}_2$ implies 
    \begin{equation} \label{New-1-1} 
    \mathcal{D}_\mathbf{w}(t)\leq \mathcal{D}_\mathbf{w}(0), \quad \forall~t > 0. 
    \end{equation}
  Next we claim that 
  \begin{align}\label{dia-x}
  \sup_{0\leq t<\infty}\mathcal{D}_{\mathbf{x}}(t)<\mathcal{D}_{\mathbf{x}}^{\infty}.
  \end{align}
  Suppose the contrary holds, i.e., the relation \eqref{dia-x} is not true. Then there exists a constant $t_0>0$ such that
  \begin{equation} \label{New-2}
  \mathcal{D}_{\mathbf{x}}(t)<\mathcal{D}_{\mathbf{x}}^{\infty},\quad \forall~t \in [0, t_0), \quad   \mathcal{D}_{\mathbf{x}}(t_0)=\mathcal{D}_{\mathbf{x}}^{\infty}.
   \end{equation}
  We use the second inequality in \eqref{SDDI-r} and \eqref{New-1-1} to obtain 
  \begin{align*}
  \begin{aligned}
    \mathcal{D}_{\mathbf{x}}(t)\leq& D_{\mathbf{x}}(0)+\frac{c^2}{c^2+1}\int_0^t \exp\left\{-\left(\frac{\phi(\mathcal{D}^{\infty}_{\mathbf{x}})}{T^*}
    -\frac{2\Lambda_2(\mathcal{D}_{\mathbf{w}}(0))}{c^{2}T^*}\right)s\right\}\mathcal{D}_{\mathbf{w}}(0)\,ds  \\
    \leq&  \mathcal{D}_{\mathbf{x}}(0)+ \frac{c^4T^*\mathcal{D}_{\mathbf{w}}(0)}{(c^2+1)[c^2\phi(\mathcal{D}^{\infty}_{\mathbf{x}})-2\Lambda_2(\mathcal{D}_{\mathbf{w}}(0))]}<\mathcal{D}_{\mathbf{x}}^{\infty}, \quad t\in [0, t_0],
    \end{aligned}
  \end{align*}
i.e.,
\[ \mathcal{D}_{\mathbf{x}}(t)<\mathcal{D}_{\mathbf{x}}^{\infty},\qquad t\in [0, t_0]. \]
  This contradicts to the defining condition of $t_0$ in \eqref{New-2}. Therefore, we have the uniform bound \eqref{dia-x}. \newline
  
\noindent (ii)~ We use  \eqref{dia-x} and  $\eqref{SDDI-r}_2$ to find 
\[
\frac{\mathrm{d}\mathcal{D}_\mathbf{w}}{\mathrm{d} t}\leq -\left(\frac{\phi(\mathcal{D}_{\mathbf{x}})}{T^*}-\frac{2\Lambda_2(\mathcal{D}_{\mathbf{w}})}{c^{2}T^*}\right)\mathcal{D}_\mathbf{w} \leq 
 - \lambda \mathcal{D}_\mathbf{w}.
\]
This implies the desired exponential decay of $D_\mathbf{w}$. 
\end{proof}
\begin{remark} \label{R3.2}
In previous literature, all authors assume more restrictive assumptions on the initial data compared to \eqref{NN-2-0}:
\begin{enumerate}
\item
In \cite{Ha-Kim-Ruggeri-ARMA-2020}, the authors assume the a priori assumption on  $\mathbf{w}$  and the communication weight $\phi$ is close to a constant:
\begin{equation*} %\label{NN-2}
 \sup_{0 \leq t < \infty} \max_{a \in [N]} |\mathbf{w}_a |^2 \leq R, \quad \frac{\min \phi_{ab}}{\max \phi_{ab}} \geq  \frac{2C(R) F(R) R}{c^2}.
 \end{equation*}
\item
In \cite{Ahn-Ha-Kim-CPAA-2021}, the authors assume that the initial data satisfy a formally more complicated and essentially similar assumption as in \eqref{NN-2-0} except for more restrictive coefficients . 
\vspace{0.2cm}
\item
In \cite{Ahn-Ha-Kim-JMP-2022}, the assumption on initial data is quite similar to that in \cite{Ahn-Ha-Kim-CPAA-2021}.
\end{enumerate}
\end{remark}

\vspace{0.5cm}

As a direct corollary of  the flocking estimates in Theorem \ref{T3.1}, we can derive the $\ell^2-$ norm estimate for $\bf{w}_a$. For this, we define a nonlinear functional $\mathcal{L}$ as
\begin{align*}
\mathcal{L}:=&\frac{1}{2N^2}\sum_{a,b=1}^N \left|\mathbf{w}_a-\mathbf{w}_b\right|^2
=\frac{1}{N}\sum_{a=1}^N \left|\mathbf{w}_a\right|^2-\frac{1}{N^2} \left|\sum_{a=1}^N\mathbf{w}_a\right|^2=\frac{1}{N}\sum_{a=1}^N \left|\mathbf{w}_a\right|^2,
\end{align*}
where we used \eqref{initialODE-sub-wt}.
\begin{corollary}\label{C3.1} Under the same conditions in Theorem \ref{T3.1} together with an extra condition $\mathcal{L}(0) <\infty$, we have
\begin{equation}\label{wc-decay}
  |\mathbf{w}_a(t)|\leq \mathcal{D}_{\mathbf{w}}(0)\mathrm{e}^{-\lambda t} \quad \mbox{and} \quad \mathcal{L}(t)\leq \mathcal{L}(0) \exp\left\{- \frac{2\Lambda_0(\mathcal{D}_\mathbf {w}(0))\phi(\mathcal{D}^{\infty}_\mathbf{x})}{T^*} t\right\}, \quad t > 0.
  \end{equation}
\end{corollary}
\begin{proof} 
\noindent (i)~The first estimate \eqref{wc-decay} follows directly from \eqref{wD-relation}  and Theorem \ref{T3.1}. \newline

\noindent (ii)~It follows from Lemma \ref{coer-F} and $\mathcal{D}_{\mathbf{x}}(t)<\mathcal{D}_{\mathbf{x}}^{\infty}$ in Theorem \ref{T3.1} that 
   \begin{align*}
\frac{\mathrm{d} \mathcal{L}}{\mathrm{d} t}&=\frac{2}{N}\sum_{a=1}^N\mathbf{w}_a\cdot\frac{\mathrm{d}}{\mathrm{d} t}\mathbf{w}_a =\frac{2}{N^2T^*}\sum_{a,b=1}^{N}\mathbf{w}_a\cdot\phi(|\mathbf{x}_a-\mathbf{x}_b|)
\left(\frac{\mathbf{w}_b}{F_b}-\frac{\mathbf{w}_a}{F_a}\right)\\
&\leq-\frac{1}{N^2T^*}\sum_{a,b=1}^{N}\phi(D_{\mathbf{x}}^{\infty})\left(\mathbf{w}_a-\mathbf{w}_b\right)\cdot
\left(\frac{\mathbf{w}_b}{F_b}-\frac{\mathbf{w}_a}{F_a}\right)\\
&\leq  -\frac{2\Lambda_0(\mathcal{D}_\mathbf{w}(0))\phi(\mathcal{D}^{\infty}_\mathbf{x})}{T^*}\mathcal{L}.
  \end{align*}
 This implies the second estimate in \eqref{wc-decay}.
  \end{proof}

% \begin{remark} Due to the exponential decay of $D_\mathbf{w}(t)$, $\sup_{0\leq t<\infty} |\mathbf{x}_{M}(t)-\mathbf{x}_{m}(t)|$ is uniformly bounded. Therefore, we can relax our assumption of the communication weight $\phi(|\mathbf{x}_a-\mathbf{x}_b|)$, we assume
%\begin{align*}
%\phi_m:=\min_{a,b}\inf_{\left|\mathbf{x}_a\right|,\left|\mathbf{x}_b\right|\leq R_0}\phi(|\mathbf{x}_a-\mathbf{x}_b|)(\left|\mathbf{x}_a-\mathbf{x}_b\right|)>0,
%\end{align*}
% where $R_0$ is a large constant depending on the initial data $\{(\mathbf{x}_{a}(0), \mathbf{w}_{a})(0)\}$.
% \end{remark} 
 %We first estimate the flocking behavior using the $\ell^2-$norm.
%

%
% With the above preparations, applying the same proof of \cite[Theorem 6.3]{Ha-Kim-Ruggeri-ARMA-2020}, we can obtain the exponential decay of the  $\ell^2-$norm.
%\begin{theorem}\label{ell2-decay}
%Let $\{(\mathbf{x}_{a}, \mathbf{w}_{a})\}$ be a global solution to \eqref{CSmodel-w} with the condition \eqref{initialODE-sub-w0} for the initial data $\{(\mathbf{x}_{a}(0), \mathbf{w}_{a})(0)\}$.
%Then, $\mathcal{L}(t)$ decays exponentially as
% $$\mathcal{L}(t)\leq \mathcal{L}(0) \mathrm{e}^{-\frac{2\Lambda(R)\phi_m}{T^*} t},$$
% where $\Lambda(R)$ is the same positive constant defined in Lemma \ref{coer-F}. Moreover, the flocking behavior emerges.
%\end{theorem}
%%%%%%%%%%%%%%%%%%%%%%%
%
%
%%%%%%%%%%%%%%%%%%%%%%%
\subsection{Uniform-time classical limit} \label{sec:3.2}
In this subsection, we focus on the classical(non-relativistic)  limit from the RCS model to the CS model, as $c\rightarrow\infty$. The optimal convergence rate of the  classical limit in an infinite-time interval will be provided in the sequel. \newline

Recall the explicit form of $\mathbf{w}_a$:
\[ \mathbf{w}_a=\mathbf{v}_a\left(1+\frac{(D_a+2)}{2\gamma^*}\Gamma_a\right)\Gamma_a, \]
and consider the following Cauchy problems:
\begin{equation}
\begin{cases}\label{CSmodel-wc}  
&\displaystyle\frac{\mathrm{d}\mathbf{x}_a}{\mathrm{d} t}=\frac{\mathbf{w}_a}{F_a},\quad a \in [N], \quad t>0,\vspace{0.2cm}\\
&\displaystyle\frac{\mathrm{d}\mathbf{w}_a}{\mathrm{d} t}
=\frac{1}{NT^*}\sum_{b=1}^{N}\phi(|\mathbf{x}_a-\mathbf{x}_b|)
\left(\frac{\mathbf{w}_b}{F_b}-\frac{\mathbf{w}_a}{F_a}\right), \vspace{0.2cm}\\
&\displaystyle\left(\mathbf{x}_a(0), \mathbf{w}_a(0)\right)=\left(\mathbf{x}^{\mathrm{in}}_a, \mathbf{w}^{\mathrm{in}}_a\right),
\end{cases}
\end{equation}
and
\begin{align}\label{CSmodel-vc}  
\begin{cases}
&\displaystyle\frac{\mathrm{d}\mathbf{x}^{\infty}_a}{\mathrm{d} t}=\mathbf{w}^{\infty}_a,\quad a \in [N], \quad t>0,\vspace{0.2cm}\\
&\displaystyle\frac{\mathrm{d}\mathbf{w}^{\infty}_a}{\mathrm{d} t}
=\frac{1}{NT^*}\sum_{b=1}^{N}
\phi(|\mathbf{x}^{\infty}_a-\mathbf{x}^{\infty}_b|)\left(\mathbf{w}^{\infty}_b-\mathbf{w}^{\infty}_a\right), \vspace{0.2cm}\\
&\displaystyle\left(\mathbf{x}^{\infty}_a(0), \mathbf{w}^{\infty}_a(0)\right)=\left(\mathbf{x}^{\infty, \mathrm{in}}_a, \mathbf{w}^{\infty,\mathrm{in}}_a \right).
\end{cases}
\end{align}
%\textcolor{red}{Note that the initial data in the above two Cauchy problems \eqref{CSmodel-wc} and \eqref{CSmodel-vc} are the same.}
As in \eqref{initialODE-sub-wt}, we have
\begin{align}\label{initialODE-sub-wvt}
\begin{aligned}
&\sum_{a=1}^N\mathbf{w}_a(t)=\sum_{a=1}^N\mathbf{w}^{\infty}_a(t)=0,  \qquad  \forall~t\geq0. 
\end{aligned}
\end{align}
To measure the difference between solutions to Cauchy problems \eqref{CSmodel-wc} and \eqref{CSmodel-vc}, we define the deviation functional $\Delta^c$ as follows:
$$\Delta^c(t):=\frac{1}{N}\sum_{a=1}^N\left(\left|\mathbf{x}_a(t)-\mathbf{x}^{\infty}_a(t)\right|^2
+\left|\mathbf{w}_a(t)-\mathbf{w}^{\infty}_a(t)\right|^2\right).$$

Next, we present flocking estimates and $\ell^2$ norm for the RCS model, which will be used in the proof of uniform-time classical limit of the kinetic RCS model.
For convenience, we first introduce the following diameter functionals:
\[ \mathcal{D}_\mathbf{x^{\infty}}:=\max_{1\leq a,b\leq N}\left|\mathbf{x}_a^{\infty}-\mathbf{x}_b^{\infty}\right|, \quad 
\mathcal{D}_\mathbf{w^{\infty}}:=\max_{1\leq a,b\leq N}\left|\mathbf{w}_a^{\infty}-\mathbf{w}_b^{\infty}\right|.
\]
We also define a Lyapunov functional $\mathcal{L}^{\infty}$ as
\begin{align*}
\mathcal{L}^{\infty}:=&\frac{1}{2N^2}\sum_{a,b=1}^N \left|\mathbf{w}^{\infty}_a(t)-\mathbf{w}^{\infty}_b\right|^2
=\frac{1}{N}\sum_{a=1}^N \left|\mathbf{w}^{\infty}_a\right|^2-\frac{1}{2N^2} \left|\sum_{a=1}^N\mathbf{w}^{\infty}_a\right|^2=\frac{1}{N}\sum_{a=1}^N \left|\mathbf{w}^{\infty}_a\right|^2.
\end{align*}
In the following lemma, we list all the flocking results for the limit system \eqref{CSmodel-vc}. 
\begin{lemma} \label{CS-fe}
\emph{ \cite{Ha-Kim-Zhang-KRM-2018}}
Let $\left\{\left(\mathbf{x}^{\infty}_a, \mathbf{w}^{\infty}_a\right)\right\}$ be a global solution to  \eqref{CSmodel-vc} with $\mathcal{L}^{\infty}(0) <\infty$. Then, the following assertions hold.
\begin{enumerate}
  \item 
 The maximal speed monotonically decreases in time:
  \[ \max_{1\leq a\leq N}|\mathbf{w}_a^{\infty}(t)|\leq \max_{1\leq a\leq N}|\mathbf{w}_a^{\infty}(s)|,\qquad 0\leq s\leq t. \]
  \vspace{0.2cm}
  \item 
  If initial data satisfy 
  \begin{align}\label{ini-CS}
  \mathcal{D}_{\mathbf{w}^{\infty}}(0)<\int_{\mathcal{D}_{\mathbf{x}^{\infty}}(0)}^{\infty}\phi(s)\, \mathrm{d}s,
  \end{align}
  then, there exists positive constant $\mathcal{D}_{\mathbf{x}^{\infty}}^{\infty}$ such that
  \[
    \sup_{0\leq t<\infty}D_{\mathbf{x}^{\infty}}(t) \leq \mathcal{D}_{\mathbf{x}^{\infty}}^{\infty},\qquad \mathcal{D}_{\mathbf{w}^{\infty}}(t)\leq \mathcal{D}_{\mathbf{w}^{\infty}}(0)\exp\left\{-\frac{\phi(\mathcal{D}_{\mathbf{x}^{\infty}}^{\infty})}{T^*}t\right\},\quad \forall ~t>0.
  \]
   \vspace{0.2cm}
  \item The functional $\mathcal{L}^{\infty}$ decays exponentially fast: 
  \begin{equation*}\label{l2winfty-decay}
  \mathcal{L}^{\infty}(t)\leq \mathcal{L}^{\infty}(0) \exp\left\{-\frac{2\phi(\mathcal{D}^{\infty}_{\mathbf{x}^{\infty}})}{T^*} t\right\},\quad \forall ~t>0.
  \end{equation*}
\end{enumerate}
\end{lemma}
As a corollary, we have the exponential decay of $ |\mathbf{w}^{\infty}_a(t)|$. 
\begin{corollary}\label{cw-limit}
Let $\left\{\left(\mathbf{x}^{\infty}_a, \mathbf{w}^{\infty}_a\right)\right\}_{a=1}^N$ be a global solution to the Cauchy problem \eqref{CSmodel-vc}. Under the condition \eqref{ini-CS}, we have
\begin{equation}\label{winfty-decay}
  |\mathbf{w}^{\infty}_a(t)|\leq \mathcal{D}_{\mathbf{w}^{\infty}}(0)\exp\left\{-\frac{\phi(\mathcal{D}_{\mathbf{x}^{\infty}}^{\infty})}{T^*} t\right\},\quad\forall~ a \in [N],\quad \forall~t>0.
  \end{equation}
\end{corollary}
\begin{proof}
The desired estimate \eqref{winfty-decay} follows directly from \eqref{initialODE-sub-wvt}  and Lemma \ref{CS-fe}.
\end{proof}
Next, we study the temporal evolutions of $\left|\mathbf{x}_a-\mathbf{x}^{\infty}_a\right|$ and $\left|\mathbf{w}_a-\mathbf{w}^{\infty}_a\right|$.
\begin{lemma}\label{L4.4}
Let $\left\{\left(\mathbf{x}_a, \mathbf{w}_a\right)\right\}$ and $\left\{\left(\mathbf{x}^{\infty}_a, \mathbf{w}^{\infty}_a\right)\right\}$ be global solutions to \eqref{CSmodel-wc} and \eqref{CSmodel-vc}, respectively. Then we have the following estimates: for $t>0$,
\begin{align*}
\begin{aligned} %\label{wMm}
&(i)\quad \frac{\mathrm{d}}{\mathrm{d} t}\left|\mathbf{x}_a-\mathbf{x}^{\infty}_a\right|\leq \left|\mathbf{w}_a-\mathbf{w}^{\infty}_a\right|+ \frac{\Lambda_1(\mathcal{D}_\mathbf{w}(0))}{c^2}\left|\mathbf{w}_a\right|,\quad a \in [N],\nonumber\\
&(ii)\quad \frac{1}{N}\frac{\mathrm{d}}{\mathrm{d} t}\sum_{a=1}^N \left|\mathbf{w}_a-\mathbf{w}^{\infty}_a\right|^2 \\
& \hspace{1.5cm} \leq -\frac{\phi(\mathcal{D}_{\mathbf{x}}^{\infty})}{NT^*}\sum_{a=1}^{N}\left|\mathbf{w}_a-\mathbf{w}^{\infty}_a\right|^2
+\frac{4\Lambda_2^2(\mathcal{D}_{\mathbf{w}}(0))\mathcal{L}(0) }{ \phi(\mathcal{D}_{\mathbf{x}}^{\infty})T^*c^{4}} \exp\left\{- \frac{2\Lambda_0(\mathcal{D}_\mathbf{w}(0))\phi(\mathcal{D}^{\infty}_\mathbf{x})}{T^*} t\right\}\\
&\hspace{1.7cm}+\frac{16\left[\phi\right]^2_{\mathrm{Lip}}\mathcal{D}^2_{\mathbf{w}^{\infty}}(0)}{ N\phi(\mathcal{D}_{\mathbf{x}}^{\infty})T^*}\sum_{a=1}^{N}|\mathbf{x}_a-\mathbf{x}^{\infty}_a|^2
\exp\left\{-\frac{2\phi(\mathcal{D}^{\infty}_{\mathbf{x}^{\infty}})}{T^*} t\right\}.
\end{aligned}
\end{align*}
Here $\Lambda_1(\mathcal{D}_\mathbf{w}(0))$ is a positive constant depending on $\mathcal{D}_\mathbf{w}(0)$.
\end{lemma}
\begin{proof} $(i)$ It follows from \eqref{CSmodel-wc} and \eqref{CSmodel-vc} that
\begin{align*}
\begin{aligned}
 \frac{1}{2}\frac{\mathrm{d}}{\mathrm{d} t}\left|\mathbf{x}_a-\mathbf{x}^{\infty}_a\right|^2 &= \left(\mathbf{x}_a-\mathbf{x}^{\infty}_a\right)\cdot \left(\frac{\mathbf{w}_a}{F_a}-\mathbf{w}^{\infty}_a\right)  \leq \left|\mathbf{x}_a-\mathbf{x}^{\infty}_a\right|\left( \left|\mathbf{w}_a-\mathbf{w}^{\infty}_a\right|+ \left|\mathbf{w}_a-\frac{\mathbf{w}_a}{F_a}\right|\right)\\
  &\leq   \left|\mathbf{x}_a-\mathbf{x}^{\infty}_a\right|\left( \left|\mathbf{w}_a-\mathbf{w}^{\infty}_a\right|+ \frac{\Lambda_1(\mathcal{D}_\mathbf{w}(0))}{c^2}\left|\mathbf{w}_a\right|\right).
  \end{aligned}
\end{align*}
This yields the desired estimate. \newline

\noindent $(ii)$ It follows from the dynamics of $\mathbf{w}_a$ and $\mathbf{w}^{\infty}_a$ that 
\begin{align} 
\begin{split} \label{wcvc}
&\frac{1}{2N}\sum_{a=1}^N \frac{\mathrm{d}}{\mathrm{d} t}\left|\mathbf{w}_a-\mathbf{w}^{\infty}_a\right|^2
=\frac{1}{N}\sum_{a=1}^N\left(\mathbf{w}_a-\mathbf{w}^{\infty}_a\right)\cdot \frac{\mathrm{d}}{\mathrm{d} t} \left(\mathbf{w}_a-\mathbf{w}^{\infty}_a\right) \\
&\hspace{1cm} =\frac{1}{N^2T^*}\sum_{a, b=1}^{N}\phi(|\mathbf{x}_a-\mathbf{x}_b|)\left(\mathbf{w}_a-\mathbf{w}^{\infty}_a\right)\cdot\left[\left(\frac{\mathbf{w}_b}{F_b}-\frac{\mathbf{w}_a}{F_a}\right)
-\left(\mathbf{w}^{\infty}_b-\mathbf{w}^{\infty}_a\right)\right ] \\
&\hspace{1.2cm}+\frac{1}{N^2T^*}\sum_{a, b=1}^{N}\left[\phi(|\mathbf{x}_a-\mathbf{x}_b|)-\phi(|\mathbf{x}^{\infty}_a-\mathbf{x}^{\infty}_b|)\right]
\left(\mathbf{w}_a-\mathbf{w}^{\infty}_a\right)\cdot\left(\mathbf{w}^{\infty}_b-\mathbf{w}^{\infty}_a\right) \\
&\hspace{1cm} =\frac{1}{N^2T^*}\sum_{a, b=1}^{N}\phi(|\mathbf{x}_a-\mathbf{x}_b|)\left(\mathbf{w}_a-\mathbf{w}^{\infty}_a\right)\cdot\left[\left(\mathbf{w}_b-\mathbf{w}^{\infty}_b\right)
-\left(\mathbf{w}_a-\mathbf{w}^{\infty}_a\right)\right]\\
&\hspace{1.2cm}+\frac{1}{N^2T^*}\sum_{a, b=1}^{N}\phi(|\mathbf{x}_a-\mathbf{x}_b|)\left(\mathbf{w}_a-\mathbf{w}^{\infty}_a\right)\cdot\left[\left(\frac{1}{F_b}-1\right)\mathbf{w}_b
-\left(\frac{1}{F_a}-1\right)\mathbf{w}_a\right] \\
&\hspace{1.2cm}+\frac{1}{N^2T^*}\sum_{a, b=1}^{N}\left[\phi(|\mathbf{x}_a-\mathbf{x}_b|)-\phi(|\mathbf{x}^{\infty}_a-\mathbf{x}^{\infty}_b|)\right]
\left(\mathbf{w}_a-\mathbf{w}^{\infty}_a\right)\cdot\left(\mathbf{w}^{\infty}_b-\mathbf{w}^{\infty}_a\right) \\
&\hspace{1cm} =\mathcal{I}_{11}+\mathcal{I}_{12} +\mathcal{I}_{13}.
\end{split}
\end{align}
In what follows, we provide estimates for ${\mathcal I}_{1i}$ one by one. \newline

\noindent $\bullet$ Case A (Estimate of $\mathcal{I}_{11}$): By the symmetry of $\phi(|\mathbf{x}_a-\mathbf{x}_b|)$ and \eqref{initialODE-sub-wvt}, we have
\begin{align*}
\mathcal{I}_{11}=&-\frac{1}{2N^2T^*}\sum_{a, b=1}^{N}\phi(|\mathbf{x}_a-\mathbf{x}_b|)\left|\left(\mathbf{w}_b-\mathbf{w}^{\infty}_b\right)
-\left(\mathbf{w}_a-\mathbf{w}^{\infty}_a\right)\right|^2\\
\leq&-\frac{\phi(\mathcal{D}_{\mathbf{x}}^{\infty})}{NT^*}\left[\sum_{a=1}^{N}\left|\mathbf{w}_a-\mathbf{w}^{\infty}_a\right|^2
-\frac{1}{N}\left|\sum_{a=1}^{N}\left(\mathbf{w}_a-\mathbf{w}^{\infty}_a\right)\right|^2\right]\\
=& -\frac{\phi(\mathcal{D}_{\mathbf{x}}^{\infty})}{NT^*}\sum_{a=1}^{N}\left|\mathbf{w}_a-\mathbf{w}^{\infty}_a\right|^2.
\end{align*}

\vspace{0.2cm}

\noindent $\bullet$ Case B (Estimate of $\mathcal{I}_{12}$):  We use the estimate for $\mathcal{I}_{11}$ to see
\begin{align}\label{I20}
\begin{aligned}
\left|\mathcal{I}_{12} \right|\leq&\frac{1}{2N^2T^*}\sum_{a, b=1}^{N}\phi(|\mathbf{x}_a-\mathbf{x}_b|)\left[\left(\mathbf{w}_a-\mathbf{w}^{\infty}_a\right)
-\left(\mathbf{w}_b-\mathbf{w}^{\infty}_b\right)\right]
\cdot\left[\left(\frac{1}{F_b}-1\right)\mathbf{w}_b
-\left(\frac{1}{F_a}-1\right)\mathbf{w}_a\right]\\
\leq&\frac{\phi(\mathcal{D}_{\mathbf{x}}^{\infty})}{4NT^*}\sum_{a=1}^{N}\left|\mathbf{w}_a-\mathbf{w}^{\infty}_a\right|^2
+\frac{2}{N^2 \phi(\mathcal{D}_{\mathbf{x}}^{\infty})T^*}\sum_{a, b=1}^{N}\left|\left(\frac{1}{F_b}-1\right)\mathbf{w}_b
-\left(\frac{1}{F_a}-1\right)\mathbf{w}_a\right|^2\\
\leq&\frac{\phi(\mathcal{D}_{\mathbf{x}}^{\infty})}{4NT^*}\sum_{a=1}^{N}\left|\mathbf{w}_a-\mathbf{w}^{\infty}_a\right|^2
+\frac{2}{N^2\phi(\mathcal{D}_{\mathbf{x}}^{\infty})T^*}\sum_{a, b=1}^{N}\left|\nabla_\mathbf{z}\left[\left(\frac{1}{F(z^2)}-1\right)\mathbf{z}\right]
\right|^2\left|\mathbf{w}_b-\mathbf{w}_a\right|^2.
\end{aligned}
\end{align}
Then, we combine \eqref{I20} and \eqref{I21}, and use Corollary \ref{C3.1} to obtain 
\begin{align*}%\label{I2}
\begin{aligned}
\left|\mathcal{I}_{12} \right|\leq&\frac{\phi(\mathcal{D}_{\mathbf{x}}^{\infty})}{4NT^*}\sum_{a=1}^{N}\left|\mathbf{w}_a-\mathbf{w}^{\infty}_a\right|^2
+\frac{2\Lambda_2^2(\mathcal{D}_{\mathbf{w}}(0))\mathcal{L}(0) }{ \phi(\mathcal{D}_{\mathbf{x}}^{\infty})T^*c^{4}} \exp\left\{- \frac{2\Lambda_0(\mathcal{D}_\mathbf{w}(0))\phi(\mathcal{D}^{\infty}_\mathbf{x})}{T^*} t\right\}.
\end{aligned}
\end{align*}

\vspace{0.2cm}

\noindent $\bullet$ Case C (Estimate of $\mathcal{I}_{13}$): We use \eqref{winfty-decay} to have
 \begin{align*}
 \left|\mathcal{I}_{13}\right|=&\frac{1}{2N^2T^*}\Big|\sum_{a, b=1}^{N}\left[\phi(|\mathbf{x}_a-\mathbf{x}_b|)-\phi(|\mathbf{x}^{\infty}_a-\mathbf{x}^{\infty}_b|)\right]
\left[\left(\mathbf{w}_b-\mathbf{w}^{\infty}_b\right)
-\left(\mathbf{w}_a-\mathbf{w}^{\infty}_a\right)\right]\cdot\left(\mathbf{w}^{\infty}_b-\mathbf{w}^{\infty}_a\right)\Big|\nonumber\\
\leq&\frac{\phi(\mathcal{D}_{\mathbf{x}}^{\infty})}{4NT^*}\sum_{a=1}^{N}\left|\mathbf{w}_a-\mathbf{w}^{\infty}_a\right|^2
+\frac{2}{N^2 \phi(\mathcal{D}_{\mathbf{x}}^{\infty})T^*}\sum_{a, b=1}^{N}\left[\phi\right]^2_{\mathrm{Lip}}\left(|\mathbf{x}_a-\mathbf{x}^{\infty}_a|+|\mathbf{x}_b-\mathbf{x}^{\infty}_b|
 \right)^2\left|\mathbf{w}^{\infty}_b-\mathbf{w}^{\infty}_a\right|^2\nonumber\\
\leq&\frac{\phi(\mathcal{D}_{\mathbf{x}}^{\infty})}{4NT^*}\sum_{a=1}^{N}\left|\mathbf{w}_a-\mathbf{w}^{\infty}_a\right|^2
+\frac{8\left[\phi\right]^2_{\mathrm{Lip}}\mathcal{D}^2_{\mathbf{w}^{\infty}}(0)}{ N\phi(\mathcal{D}_{\mathbf{x}}^{\infty})T^*}\sum_{a=1}^{N}|\mathbf{x}_a-\mathbf{x}^{\infty}_a|^2
\exp\left\{-\frac{2\phi(\mathcal{D}^{\infty}_{\mathbf{x}^{\infty}})}{T^*} t\right\}\nonumber.
\end{align*}
Now we collect all the estimates of $\mathcal{I}_{11}$,  $\mathcal{I}_{12}$, and $\mathcal{I}_{13}$ in \eqref{wcvc} to obtain  $(ii)$.
%\begin{align*}%\label{wcvc0}
%\frac{\mathrm{d}}{\mathrm{d} t}\sum_{a=1}^N \left|\mathbf{w}_a-\mathbf{w}^{\infty}_a\right|^2
%\leq& -\frac{\phi(D_{\mathbf{x}}^{\infty})}{T^*}\sum_{a=1}^{N}\left|\mathbf{w}_a-\mathbf{w}^{\infty}_a\right|^2
%+\frac{4\Lambda_2(D_{\mathbf{w}}^{\infty})\mathcal{L}^{\infty}(0) }{ \phi(D_{\mathbf{x}}^{\infty})T^*c^{4}} \mathrm{e}^{-\frac{2\Lambda_0(D_{\mathbf{x}}^{\infty})\phi(D_{\mathbf{x}}^{\infty})}{T^*} t}\\
%&+\frac{16\left[\phi\right]^2_{Lip}D^2_{\mathbf{w}^{\infty}}(0)}{ \phi(D_{\mathbf{x}}^{\infty})T^*}\sum_{a=1}^{N}|\mathbf{x}_a-\mathbf{x}^{\infty}_a|^2
%\mathrm{e}^{-\frac{2\phi(D^{\infty}_{\mathbf{x}^{\infty}})}{T^*} t}.
%\end{align*}
%Then we apply Gronwall's inequality in \eqref{wcvc0} to derive $(ii)$.
\end{proof}
Finally, we are now ready to verify the classical limit of the RCS model.
\begin{theorem}\label{T3.2}
Suppose that the initial data satisfy 
\begin{align}\label{indevi-RCS}
  \Delta^c(0)\leq \mathcal{O}(c^{-4}),  
\end{align}
and let $\left(\mathbf{x}_a, \mathbf{w}_a\right)$ and $\left(\mathbf{x}^{\infty}_a, \mathbf{w}^{\infty}_a\right)$ be global solutions to \eqref{CSmodel-wc} and \eqref{CSmodel-vc}, respectively. 
Then,  one has 
\begin{align}\label{limrat-CS}
\sup_{0 \leq t < \infty} \Delta^c(t)\leq \mathcal{O}(c^{-4}).
\end{align}
\end{theorem}
\begin{proof} For the desired estimate \eqref{limrat-CS}, we divide the proof into three steps: \newline
\begin{enumerate}
\item
First, we show \eqref{limrat-CS} holds in the time interval $[0, t_1]$  via the Gr\"{o}nwall differential inequality as in \cite{Ahn-Ha-Kim-JMP-2022}. Here $t_1$ is a constant satisfying
\begin{align*}
\frac{8\left[\phi\right]^2_{\mathrm{Lip}}\big[\mathcal{D}_{\mathbf{w}^{\infty}}(0)T^*\big]^2}{ \phi(\mathcal{D}_{\mathbf{x}}^{\infty})[\phi(\mathcal{D}^{\infty}_{\mathbf{x}^{\infty}})]^3}
\exp\left\{-\frac{2\phi(\mathcal{D}^{\infty}_{\mathbf{x}^{\infty}})}{T^*} t_1\right\}=\frac12.
\end{align*}
\vspace{0.2cm}
\item
Next, we will give a upper bound of $\displaystyle c^4N\sum_{a=1}^N\left|\mathbf{x}^c_a(t)-\mathbf{x}^{\infty}_a(t)\right|^2$ for $t\in (t_1, \infty)$.
\vspace{0.2cm}
\item
Finally,  we combine  the previous two steps to derive the exponential time-decay of $\displaystyle c^4\sum_{a=1}^N\left|\mathbf{w}^c_a(t)-\mathbf{w}^{\infty}_a(t)\right|^2$  and show that   \eqref{limrat-CS} holds for $t\in [0, \infty)$.
\end{enumerate}

\vspace{0.2cm}

In the sequel, we perform the above steps one by one. \newline

\noindent $\bullet$ Step A:~By the same method as in \cite{Ahn-Ha-Kim-JMP-2022} to prove \eqref{limrat-CS} in a finite time interval,  we can verify the validity of \eqref{limrat-CS}  for $t\in [0, t_1]$. \newline

\noindent $\bullet$ Step B:~For $t\in(t_1, \infty)$, we use \eqref{wc-decay} and $(i)$ in Lemma \ref{L4.4} to obtain
\begin{align*}%\label{wrt1}
\left|\mathbf{x}_a(t)-\mathbf{x}^{\infty}_a(t)\right|\leq&\left|\mathbf{x}_a(t_1)-\mathbf{x}^{\infty}_a(t_1)\right|
 +\int_{t_1}^t\left|\mathbf{w}_a(s)-\mathbf{w}^{\infty}_a(s)\right|\, \mathrm{d}s\\
 &+ \frac{\Lambda_1(\mathcal{D}_{\mathbf{x}}^{\infty})\mathcal{D}_{\mathbf{w}}^{\infty}(0)T^*}{c^2\phi(\mathcal{D}^{\infty}_{\mathbf{x}})
 -2\Lambda_2(\mathcal{D}_{\mathbf{w}}(0))}\exp\left\{-\left(\frac{\phi(\mathcal{D}^{\infty}_{\mathbf{x}})}{T^*}
 -\frac{2\Lambda_2(\mathcal{D}_{\mathbf{w}}(0))}{c^{2}T^*}\right) t_1\right\},\quad a \in [N].
\end{align*}
We square the above inequality and sum up the resulting relation over $a$ and use Step A to get 
\begin{align}\label{wrt1}
\frac{1}{N}\sum_{a=1}^N\left|\mathbf{x}_a(t)-\mathbf{x}^{\infty}_a(t)\right|^2\leq&\frac{2}{N}(t-t_1)^2\max_{t_1\leq s\leq t}\sum_{a=1}^N\left|\mathbf{w}_a(s)-\mathbf{w}^{\infty}_a(s)\right|^2+\mathcal{O}(c^{-4}).
\end{align}
We apply \eqref{wrt1},  $(ii)$ in Lemma \ref{L4.4}, and use the definition of $t_1$ to obtain
\begin{align*}
\begin{aligned}
& \frac{1}{N}\sum_{a=1}^N \left|\mathbf{w}_a(t)-\mathbf{w}^{\infty}_a(t)\right|^2 \\
& \hspace{1cm} \leq \frac{16\left[\phi\right]^2_{\mathrm{Lip}}\mathcal{D}^2_{\mathbf{w}^{\infty}}(0)}{ N\phi(\mathcal{D}_{\mathbf{x}}^{\infty})T^*}\max_{t_1\leq s\leq t}\sum_{a=1}^N\left|\mathbf{w}_a(s)-\mathbf{w}^{\infty}_a(s)\right|^2 \\
&\hspace{1.4cm}\times\int_{t_1}^t
2(s-t_1)^2\exp\left\{-\frac{2\phi(\mathcal{D}^{\infty}_{\mathbf{x}^{\infty}})}{T^*} s\right\}\, \mathrm{d}s+\mathcal{O}(c^{-4})\\
& \hspace{1cm}  \leq\frac{1}{2N}\max_{t_1\leq s\leq t}\sum_{a=1}^N\left|\mathbf{w}_a(s)-\mathbf{w}^{\infty}_a(s)\right|^2+\mathcal{O}(c^{-4}),\qquad \forall~ t\geq t_1.
\end{aligned}
\end{align*}
This implies 
\begin{align*}
\frac{1}{N}\max_{t_1\leq s\leq t}\sum_{a=1}^N\left|\mathbf{w}_a(s)-\mathbf{w}^{\infty}_a(s)\right|^2\leq\mathcal{O}(c^{-4}),\qquad \forall~ t\geq t_1.
\end{align*}
Then, we can further estimate \eqref{wrt1} as
\begin{align}\label{wat1}
\frac{1}{N}\sum_{a=1}^N\left|\mathbf{x}_a(t)-\mathbf{x}^{\infty}_a(t)\right|^2\leq&\mathcal{O}(c^{-4})(1+t-t_1)^2.\qquad \forall~ t\geq t_1.
\end{align}

\vspace{0.2cm}

\noindent $\bullet$ Step C:~Note that 
$$\frac{1}{N}\sum_{a=1}^N\left|\mathbf{x}_a(t)-\mathbf{x}^{\infty}_a(t)\right|^2\leq\mathcal{O}(c^{-4}),\qquad \forall~ t\in [0, t_1]. $$
Now, we combine  \eqref{wat1} and $(ii)$ in Lemma \ref{L4.4} to obtain
\begin{align*}%\label{wMm}
&\frac{1}{N}\frac{\mathrm{d}}{\mathrm{d} t}\sum_{a=1}^N \left|\mathbf{w}_a-\mathbf{w}^{\infty}_a\right|^2 \\
& \hspace{0.5cm}\leq -\frac{\phi(\mathcal{D}_{\mathbf{x}}^{\infty})}{NT^*}\sum_{a=1}^{N}\left|\mathbf{w}_a-\mathbf{w}^{\infty}_a\right|^2
+\mathcal{O}(c^{-4}) \exp\left\{-\frac{2\Lambda_0(\mathcal{D}_\mathbf{w}(0))\phi(\mathcal{D}_{\mathbf{x}}^{\infty})}{T^*} t\right\}\\
&\hspace{0.5cm} +\mathcal{O}(c^{-4})(1+|t-t_1|)^2
\exp\left\{-\frac{2\phi(\mathcal{D}^{\infty}_{\mathbf{x}^{\infty}})}{T^*} t\right\},\qquad \forall~ t\geq 0.
\end{align*}
By Gronwall's lemma, we use  \eqref{indevi-RCS} to have
\begin{align} 
\begin{aligned} \label{expo-tdecay}
&\frac{1}{N}\sum_{a=1}^N \left| \mathbf{w}_a-\mathbf{w}^{\infty}_a\right|^2 \leq \mathcal{O}(c^{-4})\\
& \times  \exp\left\{-\min\left\{\frac{\phi(\mathcal{D}_{\mathbf{x}}^{\infty})}{T^*}t, \frac{2\Lambda_0(\mathcal{D}_\mathbf{w}(0))\phi(\mathcal{D}_{\mathbf{x}}^{\infty})}{T^*} t , 
\frac{2\phi(\mathcal{D}^{\infty}_{\mathbf{x}^{\infty}})}{T^*} t-2\ln(1+|t-t_1|)\right\}\right\}.
\end{aligned}
\end{align}
On the other hand,  we use $(i)$ in Lemma \ref{L4.4} to find 
\begin{align*}%\label{wMm}
\left|\mathbf{x}_a-\mathbf{x}^{\infty}_a\right|\leq \left|x_a^{\mathrm{in}}-x_a^{\infty,\mathrm{in}}\right|+\int_0^t\left|\mathbf{w}_a-\mathbf{w}^{\infty}_a\right|\, \mathrm{d}s+ \frac{\Lambda_1(\mathcal{D}_{\mathbf{x}}^{\infty})}{c^2}\int_0^t\left|\mathbf{w}_a\right|\, \mathrm{d}s,\qquad a \in [N].
\end{align*}
Then, we use Corollary \ref{C3.1},  \eqref{indevi-RCS}, and \eqref{expo-tdecay} to find
\begin{align*}\label{xainfty-decay}
 &\frac{1}{N}\sum_{a=1}^N\left|\mathbf{x}_a(t)-\mathbf{x}^{\infty}_a(t)\right|^2\leq \frac{3}{N}\sum_{a=1}^N\left|x_a^{\mathrm{in}}-x_a^{\infty,\mathrm{in}}\right|^2 \nonumber\\
 & \hspace{1.5cm} +\frac{3}{N} \int_0^t(1+s)^{-2} ds\int_0^t(1+s)^2\sum_{a=1}^N\left|\mathbf{w}_a(s)-\mathbf{w}^{\infty}_a(s)\right|^2\, \mathrm{d}s\nonumber\\
& \hspace{1.5cm} +3 \frac{[\Lambda_1(\mathcal{D}_{\mathbf{x}}^{\infty})]^2}{Nc^4}\int_0^t(1+s)^{-2} ds\int_0^t(1+s)^2\sum_{a=1}^N\left|\mathbf{w}_a(s)\right|^2\, \mathrm{d}s\\
&\hspace{1.5cm} \leq \mathcal{O}(c^{-4})+\mathcal{O}(c^{-4})\int_0^t(1+s)^2 \exp\left\{-\min\left\{\frac{\phi(\mathcal{D}_{\mathbf{x}}^{\infty})}{T^*}s, \right. \right.\nonumber\\
&\hspace{1.5cm} \left.\left. \frac{2\Lambda_0(\mathcal{D}_\mathbf{w}(0))\phi(\mathcal{D}_{\mathbf{x}}^{\infty})}{T^*} s,
\frac{2\phi(\mathcal{D}^{\infty}_{\mathbf{x}^{\infty}})}{T^*} s-2\ln(1+|s-t_1|)\right\}\right\}\, \mathrm{d}s\nonumber\\
&\hspace{1.5cm} +\mathcal{O}(c^{-4})\int_0^t(1+s)^2 \exp\left\{-\frac{2\Lambda_0(\mathcal{D}_\mathbf{w}(0))\phi(\mathcal{D}_{\mathbf{x}}^{\infty})}{T^*} s\right\}\, \mathrm{d}s\nonumber\\
&\hspace{1.5cm} \leq\mathcal{O}(c^{-4}).\nonumber
\end{align*}
\end{proof}
\begin{remark} \label{R3.3}
Note that we have replaced the previous condition $\Delta^c(0) = 0$ in earlier works \cite{Ahn-Ha-Kim-JMP-2022, Ha-Kim-Ruggeri-ARMA-2020} on the initial data with a relaxed condition \eqref{indevi-RCS}. More precisely, in \cite{Ha-Kim-Ruggeri-ARMA-2020}, under the assumption that the initial data in the  Cauchy problems \eqref{CSmodel-wc} and \eqref{CSmodel-vc} are the same,  the authors derive a finite-in-time classical limit using a Gronwall type differential inequality:
\[ 
\frac{d}{dt} \Delta^c \leq C \Delta^c + {\mathcal O}(c^{-4}), \quad 0 \leq t \leq T.
\]
This yields
\[ \sup_{0 \leq t \leq T} \Delta^c(t) \leq  C(T)  {\mathcal O}(c^{-4}). \]
On the other hand, for a uniform-time classical limit, the authors considered solutions exhibiting flocking estimates, and they showed that for any $\varepsilon > 0$, there exists a positive constant $T_\varepsilon$ such that  
\[
\Delta^c(t) \leq \Delta^c(T_\varepsilon) + \varepsilon, \quad t \geq T_\varepsilon.
\]
This certainly yields
\[
\limsup_{c \to \infty} \sup_{0 \leq t < \infty} \Delta^c(t) \leq \lim_{c \to \infty} \sup_{0 \leq t \leq T_\varepsilon} \Delta^c(t) + \varepsilon = \varepsilon.
\]
Since $\varepsilon$ was arbitrary, we have
\[  \limsup_{c \to \infty} \sup_{0 \leq t < \infty} \Delta^c(t)  = 0. \]
Therefore, our result in Theorem \ref{T3.2}  provides a direct optimal convergence rate ${\mathcal O}(1) c^{-4}$ as in the finite-time classical limit in a weaker initial assumption \eqref{indevi-RCS}. 
\end{remark}

%%%%%%%%%%%%%%%%%%%%%
%
%
%%%%%%%%%%%%%%%%%%%%%%%

\section{The  kinetic RCS model} \label{sec:4}
\setcounter{equation}{0}
In this section, we study quantitative estimates on the flocking dynamics and the uniform-time  classical limit ($c\rightarrow \infty$) from the kinetic RCS model \eqref{RKCS} to the classical kinetic CS model \eqref{KCS}.  

For a sufficiently large $N$, the RCS model \eqref{CSmodel} can be effectively approximated by the corresponding mean-field model:
\begin{equation} \label{RKCS}
\begin{cases}
\displaystyle \partial_t f+\frac{\mathbf{w}}{F}\cdot \nabla_xf+\nabla_\mathbf{w}\cdot\left(L[f]f\right)=0,\qquad (t,\mathbf{x},\mathbf{w})\in \mathbb{R}_+\times\mathbb{R}^6, \\
\displaystyle L[f](t,\mathbf{x},\mathbf{w}):=-\int_{\mathbb{R}^6}\phi(|\mathbf{x}_*-\mathbf{x}|)\left(\frac{\mathbf{w}}{F}-\frac{\mathbf{w}_*}{F_*}\right)f(t, \mathbf{x}_*, \mathbf{w}_*)\, \mathrm{d} \mathbf{x}_*\mathrm{d} \mathbf{w}_*, \\
\displaystyle f \Big|_{t=0_+}=f^{\mathrm{in}},
\end{cases}
\end{equation} 
where $f=f(t,\mathbf{x},\mathbf{w})$ is the one-particle distribution function, $$F=F(w^2)=\left(1+\frac{D+2}{2\gamma^*}\Gamma\right)\Gamma.$$ Suppose that $f$ tends to $f^{\infty}$ in a suitable weak sense, as $c\rightarrow \infty$. Then, the limit distribution function $f^{\infty}$ satisfies the kinetic CS model:
\begin{equation} \label{KCS}
\begin{cases}
\displaystyle \partial_t f^{\infty}+\mathbf{w}\cdot \nabla_\mathbf{x}f^{\infty}+\nabla_\mathbf{w}\cdot\left(L[f^{\infty}]f^{\infty}\right)=0,\qquad (t,\mathbf{x},\mathbf{w})\in \mathbb{R}_+\times\mathbb{R}^6, \\
\displaystyle L[f^{\infty}](t,\mathbf{x},\mathbf{w}):=-\int_{\mathbb{R}^6}\phi(|\mathbf{x}_*-\mathbf{x}|)(\mathbf{w}-\mathbf{w}_*)f^{\infty}(t, \mathbf{x}_*, \mathbf{w}_*)\, \mathrm{d} \mathbf{x}_*\mathrm{d} \mathbf{w}_*, \\
\displaystyle f^{\infty} \Big |_{t=0_+}=f^{\infty,\mathrm{in}}.
\end{cases}
\end{equation}
%%%%%%%%%%%%%%%%%%%%%
%
%
%%%%%%%%%%%%%%%%%%%%%%%
\subsection{Asymptotic flocking dynamcis} \label{sec:4.1}
In this subsection, we recall \textit{a priori} balance laws and global well-posedness for the kinetic RCS model \eqref{RKCS} to be used in later sections.

Define an energy $E$ as follows:
\[ E:=c^2\left(\Gamma-1\right)+\frac{(D+2)k_B T^*}{2m} \left(\Gamma^2-\log\Gamma\right). \]
\begin{lemma} 
\emph{\cite{Ahn-Ha-Kim-JMP-2022}} \label{L5.1}
Let $f=f(t,\mathbf{x},\mathbf{w})$ be a global smooth solution to \eqref{RKCS} decaying sufficiently fast to zero in $\mathbf{w}-$variable. Then, we have the following assertions: \newline
\begin{enumerate}
  \item The conservation laws are given by
  $$\frac{\mathrm{d}}{\mathrm{d}t}\int_{\mathbb{R}^6}f\mathrm{d} \mathbf{x}\mathrm{d} \mathbf{w}=0,\qquad \frac{\mathrm{d}}{\mathrm{d}t}\int_{\mathbb{R}^6}\mathbf{w}f\mathrm{d} \mathbf{x}\mathrm{d} \mathbf{w}=0, \quad  t>0.$$
  \item Monotonicity of the energy $E$ is given as
  \begin{align*}
\frac{\mathrm{d}}{\mathrm{d}t}\int_{\mathbb{R}^6}Ef\mathrm{d} \mathbf{x}\mathrm{d} \mathbf{w}=&-\frac{1}{2}\int_{\mathbb{R}^{12}}\phi(|\mathbf{x}_*-\mathbf{x}|)\left|\frac{\mathbf{w}}{F}-\frac{\mathbf{w}_*}{F_*}\right|^2 f(\mathbf{x}_*, \mathbf{w}_*)f(\mathbf{x},\mathbf{w})\, \mathrm{d} \mathbf{x}_*\mathrm{d} \mathbf{w}_*\mathrm{d} \mathbf{x}\mathrm{d} \mathbf{w}, \quad  t>0.
\end{align*} 
\end{enumerate}
  \end{lemma} 

  \vspace{0.5cm}
 Before we state a global well-posedness of  \eqref{RKCS} with a compactly supported initial datum, we define $\mathbf{x}-$support, $\mathbf{w}-$support, and their diameters as follows:
  \begin{align*}
  & \Omega_\mathbf{x}(f(t)):= \mathrm{supp}_\mathbf{x}f(t,\cdot,\cdot),\hspace{2.1cm} \Omega_\mathbf{w}(f(t)):= \mathrm{supp}_\mathbf{w}f(t,\cdot,\cdot), \\
  &\Omega_\mathbf{v}(f(t)):= \left\{\frac{\mathbf{w}}{F}: \mathbf{w}\in \Omega_\mathbf{w}(f(t))\right\},\qquad\mathcal{D}_\mathbf{x}(t):= \sup_{\mathbf{x},\mathbf{x}_*\in \Omega_\mathbf{x}(f(t))}|\mathbf{x}_*-\mathbf{x}|,\\
  & \mathcal{D}_\mathbf{w}(t):= \sup_{\mathbf{w},\mathbf{w}_*\in \Omega_\mathbf{w}(f(t))}|\mathbf{w}_*-\mathbf{w}|,\hspace{1.2cm} \mathcal{D}_\mathbf{v}(t):= \sup_{\mathbf{v},\mathbf{v}_*\in \Omega_\mathbf{v}(f(t))}|\mathbf{v}_*-\mathbf{v}|.
  \end{align*} 
In the following proposition, we recall a global existence and uniqueness result of \eqref{RKCS}.
\begin{proposition} 
\emph{\cite{Ahn-Ha-Kim-JMP-2022}}
Suppose that the initial datum $f^{in}\in C^1(\mathbb{R}^6)$ has a compact support. Then, for any fixed $T\in (0, \infty)$, there exists a unique classical solution $f\in C^1([0, T)\times\mathbb{R}^6)$ to  \eqref{RKCS}.
\end{proposition}
Next, we study asymptotic flocking behaviors of the kinetic RCS model \eqref{RKCS}.
First, we recall the following simple definition of strong flocking:
\[ \mbox{Strong flocking estimates}\quad \Longleftrightarrow \quad \sup_{0\leq t<\infty}\mathcal{D}_\mathbf{x}(t)<\infty, \quad \sup_{ t\rightarrow\infty}\mathcal{D}_\mathbf{w}(t)=0. \]
For a given $(\mathbf{x}, \mathbf{v}) \in \bbr^6$, we introduce particle trajectory $[\mathbf{x}(t) := \mathbf{x}(t;0,\mathbf{x},\mathbf{v}), \mathbf{w}(t) := \mathbf{w}(t;0,\mathbf{x},\mathbf{v})]$:
\begin{equation}
\begin{cases} \label{chara-RKCS}
\displaystyle \frac{\mathrm{d}\mathbf{x}(t)}{\mathrm{d}t}=\frac{\mathbf{w}(t)}{F(t)},\qquad t>0, \\
\displaystyle \frac{\mathrm{d}\mathbf{w}(t)}{\mathrm{d}t}=\int_{\mathbb{R}^6}\phi(|\mathbf{x}_*-\mathbf{x}|)\left(\frac{\mathbf{w}_*}{F_*}-\frac{\mathbf{w}}{F}\right)f(t, \mathbf{x}_*, \mathbf{w}_*)\, \mathrm{d} \mathbf{x}_*\mathrm{d} \mathbf{w}_*,\\
\displaystyle (\mathbf{x}(0),\mathbf{w}(0))=(\mathbf{x},\mathbf{w}).
\end{cases}
\end{equation}
Note that by the definition of $\mathcal{D}_\mathbf{w}(t)$, $\lim_{ t\rightarrow\infty}\mathcal{D}_\mathbf{w}(t)=0$ means one-point concentration in the $\mathbf{w}-$ variable. As in Subsection \ref{sec:3.1}, we study the strong flocking behaviors through the exponential decay of $\mathcal{D}_\mathbf{w}(t)$ and uniform boundedness of $\mathcal{D}_\mathbf{x}(t)$. As in  Lemma \ref{lemma-SDDI} for the RCS model, we can also derive  the corresponding  SDDI for $\mathcal{D}_\mathbf{x}$ and $\mathcal{D}_\mathbf{w}$.

\begin{lemma}\label{lemma-SDDI2}  
Let $f$ be a global smooth solution to \eqref{RKCS} with a compactly supported initial datum $f^{\mathrm{in}}$ such that
$$\int_{\mathbb{R}^6}f^{\mathrm{in}}\, \mathrm{d} \mathbf{x}\mathrm{d} \mathbf{w}=1, \qquad \int_{\mathbb{R}^6}\mathbf{w}f^{\mathrm{in}}\, \mathrm{d} \mathbf{x}\mathrm{d} \mathbf{w}=0.$$
Then the diameter pair $(\mathcal{D}_\mathbf{x}, \mathcal{D}_\mathbf{w})$ satisfies the following differential inequalities: 
\begin{equation} 
\begin{cases} \label{SDDI-rk}
\displaystyle \frac{\mathrm{d}\mathcal{D}_\mathbf{x}}{\mathrm{d} t}\leq\frac{c^2}{c^2+1}\mathcal{D}_\mathbf{w}, \quad t > 0, \vspace{0.2cm}\\
\displaystyle \frac{\mathrm{d}\mathcal{D}_\mathbf{w}}{\mathrm{d} t}\leq -\left(\phi(\mathcal{D}_{\mathbf{x}})-\frac{2\Lambda_2(\mathcal{D}_\mathbf{w})}{c^{2}}
\right)\mathcal{D}_\mathbf{w}.
\end{cases}
\end{equation}
\end{lemma}
\begin{proof}  
It follows from $(i)$ in Lemma \ref{L5.1} that 
$$\int_{\mathbb{R}^6}f\, \mathrm{d} \mathbf{x}\mathrm{d} \mathbf{w}=1, \qquad \int_{\mathbb{R}^6}\mathbf{w}f\, \mathrm{d} \mathbf{x}\mathrm{d} \mathbf{w}=0,\qquad t\geq0.$$
Now we derive the two inequalities in \eqref{SDDI-rk} seperately. \newline

\noindent $\bullet$~Case A (Derivation of differential inequality for $\mathcal{D}_\mathbf{x})$:~we take two particle trajectories $[\mathbf{x}_1(t), \mathbf{w}_1(t)]$ and $[\mathbf{x}_2(t), \mathbf{w}_2(t)]$ such that
  $$|\mathbf{w}_1(t)-\mathbf{w}_2(t)|=\mathcal{D}_\mathbf{w}(t).$$
  We use \eqref{chara-RKCS} and \cite[Corollary $2.1$]{Ahn-Ha-Kim-JMP-2022} to obtain
  \begin{align*}
  \begin{aligned}
  \frac{\mathrm{d}\mathcal{D}^2_\mathbf{x}}{\mathrm{d} t} &=2(\mathbf{x}_1-\mathbf{x}_2)\cdot\frac{\mathrm{d}(\mathbf{x}_1-\mathbf{x}_2)}{\mathrm{d} t}=2(\mathbf{x}_1-\mathbf{x}_2)\cdot\left(\frac{\mathbf{w}_1}{F_1}-\frac{\mathbf{w}_2}{F_2}\right) \leq 2|\mathbf{x}_1-\mathbf{x}_2|\left|\frac{\mathbf{w}_1}{F_1}-\frac{\mathbf{w}_2}{F_2}\right| \\
  &\leq\frac{2c^2}{c^2+1}\mathcal{D}_\mathbf{x}|\mathbf{w}_1-\mathbf{w}_2|.
\end{aligned}
  \end{align*}
  This implies the  first inequality in \eqref{SDDI-rk}.  \newline
  
\noindent $\bullet$~Case B (Derivation of differential inequality for $\mathcal{D}_\mathbf{w}$):~we take two particle trajectories $[\mathbf{x}_1(t), \mathbf{w}_1(t)]$ and $[\mathbf{x}_2(t), \mathbf{w}_2(t)]$ such that
 \begin{align}\label{wdiam-rkcs}
 |\mathbf{w}_1(t)-\mathbf{w}_2(t)|=\mathcal{D}_\mathbf{w}(t).
 \end{align}
  We use \eqref{chara-RKCS} to have
  \begin{align*} 
  \begin{aligned} %\label{w-krcs}
    \frac{\mathrm{d}\mathcal{D}^2_\mathbf{w}}{\mathrm{d} t} &=  2(\mathbf{w}_1-\mathbf{w}_2)\cdot\frac{\mathrm{d}(\mathbf{w}_1-\mathbf{w}_2)}{\mathrm{d} t} \\
    & =2\int_{\mathbb{R}^6}\phi(|\mathbf{x}_*-\mathbf{x}_1|)\left(\frac{\mathbf{w}_*}{F_*}-\frac{\mathbf{w}_1}{F_1}\right)\cdot (\mathbf{w}_1-\mathbf{w}_2)f_*\, \mathrm{d} \mathbf{x}_*\mathrm{d} \mathbf{w}_*\\
    &-2\int_{\mathbb{R}^6}\phi(|\mathbf{x}_*-\mathbf{x}_2|)\left(\frac{\mathbf{w}_*}{F_*}-\frac{\mathbf{w}_2}{F_2}\right)\cdot (\mathbf{w}_1-\mathbf{w}_2)f_*\, \mathrm{d} \mathbf{x}_*\mathrm{d} \mathbf{w}_* \\
    &=2\int_{\mathbb{R}^6}\phi(|\mathbf{x}_*-\mathbf{x}_1|)(\mathbf{w}_*-\mathbf{w}_1)\cdot (\mathbf{w}_1-\mathbf{w}_2)f_*\, \mathrm{d} \mathbf{x}_*\mathrm{d} \mathbf{w}_*\\
    &-2\int_{\mathbb{R}^6}\phi(|\mathbf{x}_*-\mathbf{x}_2|)(\mathbf{w}_*-\mathbf{w}_2)\cdot (\mathbf{w}_1-\mathbf{w}_2)f_*\, \mathrm{d} \mathbf{x}_*\mathrm{d} \mathbf{w}_* \\
    &+2\int_{\mathbb{R}^6}\phi(|\mathbf{x}_*-\mathbf{x}_1|)\left[\left(\frac{1}{F_*}-1\right)\mathbf{w}_*-\left(\frac{1}{F_1}-1\right)\mathbf{w}_1\right]
    \cdot (\mathbf{w}_1-\mathbf{w}_2)f_*\, \mathrm{d} \mathbf{x}_*\mathrm{d} \mathbf{w}_* \\
    &-2\int_{\mathbb{R}^6}\phi(|\mathbf{x}_*-\mathbf{x}_2|)\left[\left(\frac{1}{F_*}-1\right)\mathbf{w}_*-\left(\frac{1}{F_1}-1\right)\mathbf{w}_2\right]
    \cdot (\mathbf{w}_1-\mathbf{w}_2)f_*\, \mathrm{d} \mathbf{x}_*\mathrm{d} \mathbf{w}_*.
   \end{aligned} 
  \end{align*}
  By the choice of the characteristic curves with \eqref{wdiam-rkcs},  we have
  \begin{align*}
        [\mathbf{w}_*-\mathbf{w}_1(t)]\cdot[\mathbf{w}_1(t)-\mathbf{w}_2(t)]\leq 0,\qquad 
         &[\mathbf{w}_*-\mathbf{w}_2(t)]\cdot[\mathbf{w}_1(t)-\mathbf{w}_2(t)]\geq 0
         \end{align*}
   for any $\mathbf{w}_*\in \Omega_\mathbf{w}(f(t))$, and 
        $$\phi(|\mathbf{x}_*-\mathbf{x}_1(t)|)\geq \phi(\mathcal{D}_{\mathbf{x}}(t)),\qquad \phi(|\mathbf{x}_*-\mathbf{x}_2(t)|)\geq \phi(\mathcal{D}_{\mathbf{x}}(t)).$$
  We can use \eqref{I21} to see
  \begin{align*}
  \begin{aligned} %\label{wd-krcs}
    &\frac{\mathrm{d}\mathcal{D}^2_\mathbf{w}}{\mathrm{d} t}-2\phi(\mathcal{D}_{\mathbf{x}})\mathcal{D}^2_\mathbf{w} \\
    &\hspace{0.5cm} \leq \frac{2\Lambda_2(\mathcal{D}_\mathbf{w})}{c^{2}}\int_{\mathbb{R}^6}\phi(|\mathbf{x}_*-\mathbf{x}_1|)|\mathbf{w}_*-\mathbf{w}_1|
    |\mathbf{w}_1-\mathbf{w}_2|f_*\, \mathrm{d} \mathbf{x}_*\mathrm{d} \mathbf{w}_* \\
    &\hspace{0.5cm}  +\frac{2\Lambda_2(\mathcal{D}_\mathbf{w})}{c^{2}}\int_{\mathbb{R}^6}\phi(|\mathbf{x}_*-\mathbf{x}_2|)|\mathbf{w}_*-
\mathbf{w}_2|
    |\mathbf{w}_1-\mathbf{w}_2|f_*\, \mathrm{d} \mathbf{x}_*\mathrm{d} \mathbf{w}_* \\
    &\hspace{0.5cm}  \leq \frac{2\Lambda_2(\mathcal{D}_\mathbf{w})}{c^{2}}\left(|\mathbf{w}_*-\mathbf{w}_1|+|\mathbf{w}_*-\mathbf{w}_2|\right)|\mathbf{w}_1-\mathbf{w}_2| \\
    &\hspace{0.5cm}  \leq \frac{4\Lambda_2(\mathcal{D}_\mathbf{w})}{c^{2}}\mathcal{D}^2_\mathbf{w}. 
\end{aligned}
\end{align*}
This implies \eqref{SDDI-rk}.
  \end{proof}

With the above preparations, we can provide the strong flocking estimate of \eqref{RKCS}.
\begin{theorem}\label{T4.1}
 Under the same conditions in Lemma \ref{lemma-SDDI2}, we further assume that 
 \begin{align*} 
 \phi(\mathcal{D}_\mathbf{x}^{\infty})>\frac{2\Lambda_2(D_\mathbf{w}(0))}{c^2},\qquad \mathcal{D}_\mathbf{x}(0)+\frac{c^4\mathcal{D}_\mathbf{w}(0)}{(c^2+1)[c^2\phi(\mathcal{D}_\mathbf{x}^{\infty})-2\Lambda_2(D_\mathbf{w}(0))]}
      <\mathcal{D}_\mathbf{x}^{\infty}.
 \end{align*}
Then, we have
 $$\sup_{0\leq t<\infty}\mathcal{D}_\mathbf{x}(t)<\mathcal{D}_\mathbf{x}^{\infty},\qquad \mathcal{D}_\mathbf{w}(t)\leq \mathcal{D}_\mathbf{w}(0) \exp\left\{-\left(\phi(\mathcal{D}_\mathbf{x}^{\infty})-\frac{2\Lambda_2(D_\mathbf{w}(0))}{c^2}\right) t\right\}, \quad t \in [0, \infty).$$
   \end{theorem}
  \begin{proof} We omit details since it is almost the same as that in Theorem \ref{T3.1}. 
  \end{proof}
  \begin{remark} \label{R4.1}
   In \cite{Ahn-Ha-Kim-JMP-2022}, the authors imposed similar assumptions as those in Theorem \ref{T4.1} with more restrictive coefficients in the relationship between $\mathcal{D}_\mathbf{x}$ and $\mathcal{D}_\mathbf{w}$.
\end{remark}

\vspace{0.5cm}
\subsection{Uniform-time classical limit} \label{sec:4.2}
In this subsection, we focus on the uniform-time classical limit from \eqref{RKCS} to \eqref{KCS} with an optimal convergence rate. Different from the particle model in the previous section, we will proceed the   classical limit of the kinetic equation by employing the uniform-time  mean-field estimates in the framework of measure-valued solutions. \newline

Let $\mathcal{P}(\mathbb{R}^6)$ be the set of all probability Radon measures on  $\mathbb{R}^6$, which can be identified with the set of normalized non-negative bounded linear functionals on $\mathcal{C}_0(\mathbb{R}^6)$. For any $\mu\in\mathcal{P}(\mathbb{R}^6)$ and $\varphi\in \mathcal{C}_0(\mathbb{R}^6)$, we have the following standard duality relation:
$$(\mu, \varphi):=\int_{\mathbb{R}^6}\varphi(\mathbf{x},\mathbf{w})\mu(\mathrm{d}\mathbf{x}, \mathrm{d}\mathbf{w}).$$

As in \cite{Ahn-Ha-Kim-CPAA-2021}, we recall the concept of measure-valued solutions to \eqref{RKCS} as follows.
\begin{definition}\cite{Ahn-Ha-Kim-CPAA-2021}\label{mvs-rkcs}
For any $T\in [0, \infty)$, a time-dependent probability Radon measure $\mu \in L^{\infty}([0,T);\mathcal{P}(\mathbb{R}^6))$ is a measure-valued solution to \eqref{RKCS} with the initial datum $\mu_0 \in \mathcal{P}(\mathbb{R}^6)$ if the following assertions hold:
  \begin{enumerate}
    \item For any $\varphi\in \mathcal{C}_0([0,T)\times\mathbb{R}^6)$, $\mu$~ is weakly continuous in $t$:
    $$t\mapsto (\mu_t, \varphi) \mbox{is continuous in t}.$$ 
    \item For any $\psi\in \mathcal{C}^1_0([0,T)\times\mathbb{R}^6)$, $\mu$ satisfies \eqref{RKCS} in a weak sense:
        $$(\mu_t, \psi(t, \cdot,\cdot))-(\mu_0, \psi(0, \cdot,\cdot))=\int_0^t\left(\mu_s, \partial_s\psi+\frac{\mathbf{w}}{F}\cdot \nabla_\mathbf{x}\psi+L[\mu_s]\cdot\nabla_\mathbf{w}\psi\right)\mathrm{d}s,$$
    where $L[\mu_s]$ takes the explicit form:
    $$L[\mu_s](\mathbf{x},\mathbf{w}):=-\int_{\mathbb{R}^6}\phi(|\mathbf{x}_*-\mathbf{x}|)\left(\frac{\mathbf{w}_*}{F_*}-\frac{\mathbf{w}}{F}\right)f(t_*, \mathbf{x}_*, \mathbf{w}_*)\mu_t( \mathrm{d} \mathbf{x}_*,\mathrm{d} \mathbf{w}_*).$$    
        \end{enumerate}
\end{definition}
Next, we recall the Wasserstein distance between measures in $\mathcal{P}(\mathbb{R}^6)$.
\begin{definition}\cite{Villani-2009} For $p\in [1, \infty)$, let $\mathcal{P}_{p}(\mathbb{R}^6)$ be a collection of all probability Radon measures with the finite $p-$th  moment:
$$\mathcal{P}_{p}(\mathbb{R}^6):=\left\{\mu \in \mathcal{P}(\mathbb{R}^6): \int_{\mathbb{R}^6}|z-z_0|^p\mu(\mathrm{d}z)<\infty \quad \mbox{for some  } z_0\in \mathbb{R}^6\right\}.$$
Then, the $p-$Wasserstein distance $W_p(\mu, \nu)$ is defined for any $\mu, \nu \in \mathcal{P}_{p}(\mathbb{R}^6)$ as
$$W^p_p(\mu, \nu):=\inf_{\gamma\in\Gamma(\mu, \nu)}\left(\int_{\mathbb{R}^6\times\mathbb{R}^6}|z-z_*|^p\,\mathrm{d}\gamma(z,z_*)\right),$$
where $\Gamma(\mu, \nu)$ denotes the collection of all probability Radon measure on $\mathbb{R}^6\times\mathbb{R}^6$ with marginals $\mu$ and $\nu$.
\end{definition}

\vspace{0.5cm}

It is well known from \cite{Villani-2009} that for any initial measure $\mu_0$ with a compact support, there exists a sequence of empirical measures $\left(\mu_0^N\right)_{N=1}^{\infty}$ on $\mathbb{R}^6$ defined by
$$\mu_0^N(\mathrm{d}\mathbf{x}, \mathrm{d}\mathbf{w}):= \frac{1}{N}\sum_{a=1}^N\delta_{(\mathbf{x}^{0}_a, \mathbf{w}^{0}_a)}$$
such that 
$$\limsup_{N\rightarrow\infty}W_1(\mu_0^N, \mu_0)=0.$$
Let $\mu_t^{N,\infty}(\mathrm{d}\mathbf{x}^\infty, \mathrm{d}\mathbf{w}^\infty)$ be the empirical measure defined by
\begin{align}\label{em-muinf}
\mu_t^{N,\infty}(\mathrm{d}\mathbf{x}^\infty, \mathrm{d}\mathbf{w}^\infty):=\frac{1}{N}\sum_{a=1}^N\delta_{(\mathbf{x}^\infty_a(t), \mathbf{w}^\infty_a(t))},
\end{align}
where $\left\{(\mathbf{x}^\infty_a(t), \mathbf{w}^\infty_a(t))\right\}_{a=1}^N$ is a solution to \eqref{CSmodel-vc} with the initial data $\{(\mathbf{x}^{\mathrm{in},\infty}_a, \mathbf{w}^{\mathrm{in},\infty}_a)\}_{a=1}^N$. Correspondingly, we can also define the empirical measure $\mu_t^{N}(\mathrm{d}\mathbf{x}, \mathrm{d}\mathbf{w})$ as
\begin{align}\label{em-mu}
\mu_t^{N}(\mathrm{d}\mathbf{x}, \mathrm{d}\mathbf{w}):=\frac{1}{N}\sum_{a=1}^N\delta_{(\mathbf{x}_a(t), \mathbf{w}_a(t))},
\end{align}
where $\left\{(\mathbf{x}_a(t), \mathbf{w}_a(t))\right\}_{a=1}^N$ is a solution to \eqref{CSmodel-wc} with the same initial data $\{(\mathbf{x}^{\mathrm{in},\infty}_a, \mathbf{w}^{\mathrm{in},\infty}_a)\}_{a=1}^N$. Then, the empirical measures $\mu^N_t$ and $\mu^{N,\infty}_t$ are measure-valued solutions to  \eqref{RKCS} and  \eqref{KCS}, respectively.
The following uniform-time mean-field limit results imply that the measure-valued solutions  $\mu_t$ and $\mu^{\infty}_t$ of  \eqref{RKCS} and  \eqref{KCS} can be uniformly approximated by the empirical measures $\mu_t^{N}$ and $\mu_t^{N,\infty}$, respectively, under suitable assumptions on the initial measure $\mu_0$.

\begin{proposition} 
\emph{\cite{Ahn-Ha-Kim-CPAA-2021,Ha-Kim-Zhang-KRM-2018}} \label{KCS-lim} The following assertions hold:
\begin{enumerate}
  \item Suppose that the initial measure $\mu_0$ is compactly supported and satisfies the following condition: there exist a positive constant $\mathcal{D}_\mathbf{x}^{\infty}$ such that
      \begin{align*} 
      \mathcal{D}_\mathbf{x}(0)+\frac{c^4\mathcal{D}_\mathbf{w}(0)}{(c^2+1)[c^2\phi(\mathcal{D}_\mathbf{x}^{\infty})-2\Lambda_2(D_\mathbf{\mathbf{w}}(0))]}
      <\mathcal{D}_\mathbf{x}^{\infty},
 \end{align*}
 where $\mathcal{D}_\mathbf{x}(0)$, and $\mathcal{D}_\mathbf{w}(0)$ denote the diameters of the $\mathbf{x}-$support and $\mathbf{w}-$support of the initial measure $\mu_0$, respectively. Let $\mu_t$ be a global measure-valued solution to \eqref{RKCS} with initial datum $\mu_0$ and $\mu_t^{N}$ be a measure defined in \eqref{em-mu}.
Then, one has
$$\limsup_{N\rightarrow\infty}W_1(\mu_t^N, \mu_t)=0.$$  
  \item Suppose that the initial measure $\mu_0$ is compactly supported and satisfies the following condition:
  $$\mathcal{D}_\mathbf{x}(0)+\mathcal{D}_\mathbf{w}(0)<\infty, \qquad \mathcal{D}_\mathbf{w}(0)<\int_{\mathcal{D}_\mathbf{x}(0)}^{\infty}\phi(s)\,\mathrm{d}s,  $$
  and let $\mu_t^{\infty}$ be a measure-valued solution to \eqref{KCS} with initial datum $\mu_0$, and $\mu_t^{N,\infty}$ be the measure defined in \eqref{em-muinf}. Then, one has
$$\limsup_{N\rightarrow\infty}W_1(\mu_t^{N,\infty}, \mu^{\infty}_t)=0.$$ 
\end{enumerate} 
  
\end{proposition} 
\begin{remark}
The initial conditions in \cite{Ahn-Ha-Kim-CPAA-2021} are more complicated and more restrictive. Here we replace it with a much simpler one. 
\end{remark}
Now we are ready to prove the uniform-time classical limit for the kinetic RCS model \eqref{RKCS} with an optimal convergence rate.
\begin{theorem} \label{T4.2}
Suppose that the initial probability Radon measure $\mu_0$ satisfies the conditions in Proposition \ref{KCS-lim}, and let $\mu_t$ and $\mu_t^{\infty}$ be measure-valued solutions to Cauchy problems \eqref{RKCS} and \eqref{KCS} with the compactly supported initial measures $\mu_0$ and $\mu_0^{\infty}$ satisfying
\begin{align}\label{indevi-kRCS}
    W_1(\mu_0, \mu^{\infty}_0) \leq \mathcal{O}(c^{-2}).
\end{align}
Then, we have 
\begin{align}\label{w1-cr}
\sup_{0\leq t<\infty}W_1(\mu_t, \mu^{\infty}_t) \leq \mathcal{O}(c^{-2}).
\end{align}
\end{theorem}
\begin{proof}
The $1-$Wasserstein distance between $\mu_t$ and $\mu_t^{\infty}$ can be estimated as
\begin{align}\label{inin}
W_1(\mu_t, \mu^{\infty}_t)\leq W_1(\mu_t, \mu^{N}_t)+W_1(\mu^N_t, \mu^{N,\infty}_t)+W_1(\mu^{N,\infty}_t, \mu^{\infty}_t).
\end{align}
Note that the transport plan from $\mu^N_t$ to $\mu^{N,\infty}_t$ can be encoded as the $N\times N$ matrix $\mathbb{M}=(M_{ab})$ with the following constraint:
$$\sum_{a=1}^N M_{ab}=\sum_{b=1}^N M_{ab}=\frac{1}{N}.$$
Considering the particular transport plan $M_{ab}=\frac{1}{N}\delta_{ab}$, we use \eqref{limrat-CS} in Theorem \ref{T3.2} to get
\begin{align*}%\label{nn}
W_2^2(\mu^N_t, \mu^{N,\infty}_t)\leq \frac{1}{N}\sum_{a=1}^N\left(\left|\mathbf{x}_a(t)-\mathbf{x}^{\infty}_a(t)\right|^2
+\left|\mathbf{w}_a(t)-\mathbf{w}^{\infty}_a(t)\right|^2\right)= \Delta^c(t) .
\end{align*}
This together with the monotonicity of the $p-$Wasserstein distance with respect to $p$ implies
\begin{align}\label{nn}
W_1(\mu^N_t, \mu^{N,\infty}_t)\leq  W_2(\mu^N_t, \mu^{N,\infty}_t)\leq \sqrt{\Delta^c(t)}=\mathcal{O}(c^{-2}).
\end{align}
Then, by the uniform-time mean-field limits presented in Proposition \ref{KCS-lim}, we take $N\rightarrow\infty$ in \eqref{inin} and use \eqref{nn} to derive \eqref{w1-cr}.
\end{proof}
\begin{remark} \label{R4.2}
In \cite{Ahn-Ha-Kim-JMP-2022}, instead of \eqref{indevi-kRCS}, the authors assume that the initial measures $\mu_0$ and $\mu_0^{\infty}$ are the same, and define $\widetilde{\Delta}^c(t)$ as
$$\widetilde{\Delta}^c(t):=\sum_{a=1}^N\left(\left|\mathbf{x}_a(t)-\mathbf{x}^{\infty}_a(t)\right|^2
+\left|\mathbf{w}_a(t)-\mathbf{w}^{\infty}_a(t)\right|^2\right).$$
Correspondingly, instead of \eqref{nn}, they obtain
\begin{align*}%\label{nn0}
W_1(\mu^N_t, \mu^{N,\infty}_t)\leq W_2(\mu^N_t, \mu^{N,\infty}_t)\leq \sqrt{\frac{\widetilde{\Delta}^c(t)}{N}}=\mathcal{O}(N^{-\frac{1}{2}}).
\end{align*}
Then, it follows from \eqref{inin} and the uniform-time mean-field limits presented in Proposition \ref{KCS-lim} that
$$\lim_{c\rightarrow\infty}\sup_{0\leq t<\infty} W_1(\mu_t, \mu^{\infty}_t)=0. $$
Namely, they derive a uniform-time classical limit without convergence rate  from \eqref{RKCS} to \eqref{KCS} under a stronger assumption for initial data.
\end{remark}

\section{Conclusion} \label{sec:5}
\setcounter{equation}{0}
In this paper, we present improved flocking estimates and a uniform-time mean-field model for the TCS model in the simplified mechanical  case. In \cite{Ha-Kim-Ruggeri-ARMA-2020}, the authors developed a relativistic TCS model and analyzed its flocking dynamics as well as the uniform-time mean-field limit, transitioning from the relativistic CS model to the classical CS model. Our study provides sufficient conditions for these dynamics in terms of initial data and system parameters.
Under the proposed conditions, we demonstrate that flocking dynamics emerge asymptotically: the spatial diameter remains uniformly bounded over time, while the velocity diameter converges to zero asymptotically. Additionally, we derive an optimal convergence rate from the relativistic model to the classical model, which is of order $c^{-4}$, consistent with the finite-time classical limit.
However, several important aspects remain unexplored in this work. For instance, we focused on the simplified relativistic mechanical framework rather than addressing the full relativistic thermodynamical case. This intriguing extension is left for future investigation.

%\textbf{Acknowledgments:} The work of S.-Y. Ha is The work of Q. H. Xiao is supported by the National Natural Science Foundation of China under contract  , and the work of  is supported by .

\end{document}